\documentclass[11pt]{article}
\usepackage[english]{babel}
\usepackage{times}
\usepackage{draftcopy}
\usepackage{paralist}
\usepackage{amsmath,amstext,amssymb,amsthm, amsfonts}
\usepackage[dvips]{graphics,color}
\newcommand{\EndPf}{\hfill $\square$ \medskip}     
\newcommand{\bi}{\begin{itemize}}  
\newcommand{\ei}{\end{itemize}}     
\newcommand{\bc}{\begin{center}}  
\newcommand{\ec}{\end{center}}     

\newcommand{\ls}[1]
   {\dimen0=\fontdimen6\the\font \lineskip=#1\dimen0
   \advance\lineskip.5\fontdimen5\the\font \advance\lineskip-\dimen0
   \lineskiplimit=.9\lineskip \baselineskip=\lineskip
   \advance\baselineskip\dimen0 \normallineskip\lineskip
   \normallineskiplimit\lineskiplimit \normalbaselineskip\baselineskip
   \ignorespaces }

\numberwithin{equation}{section}

\newcommand{\slim} {\mathop{\rm lim\,sup}}
\newcommand{\ilim} {\mathop{\rm lim\,inf}}

\newtheorem{lemma}{Lemma}[section]
\newtheorem{theorem}[lemma]{Theorem}
\newtheorem{corollary}[lemma]{Corollary}

\newtheorem{example}[lemma]{Example}

\newtheorem{remark}[lemma]{Remark}

\def\D{\mathcal{D}}
\def\S{\mathbb{S}}
\def\K{\mathbb{K}}
\def\Y{\mathbb{Y}}
\def\h{\mathbf{I}}
\def\X{\mathbb{X}}
\def\E{\mathbb{E}}
\def\A{\mathbb{A}}
\def\H{\mathbb{H}}
\def\R{\overline{\mathbb{R}}}
\def\P{\mathbb{P}}
\def\F{\mathbb{F}}
\def\T{\mathbb{T}}
\def\C{\mathbb{C}}
\def\W{\mathbb{W}}
\def\c{\bar{c}}
\def\B{\mathcal{B}}
\def\oo{\mathcal{O}}
\def\g{\mathcal{G}}

\title{Partially Observable Total-Cost Markov Decision Processes with Weakly Continuous Transition Probabilities}

\begin{document}

\maketitle

\begin{center}
Eugene~A.~Feinberg \footnote{Department of Applied Mathematics and
Statistics,
 Stony Brook University,
Stony Brook, NY 11794-3600, USA, eugene.feinberg@sunysb.edu},\
Pavlo~O.~Kasyanov\footnote{Institute for Applied System Analysis,
National Technical University of Ukraine ``Kyiv Polytechnic
Institute'', Peremogy ave., 37, build, 35, 03056, Kyiv, Ukraine,\
kasyanov@i.ua.},\ and Michael~Z.~Zgurovsky\footnote{National Technical University of Ukraine
``Kyiv Polytechnic Institute'', Peremogy ave., 37, build, 1, 03056,
Kyiv, Ukraine,\
 zgurovsm@hotmail.com
}\\

\bigskip
\end{center}

\begin{abstract}
This paper describes sufficient conditions for the existence of
optimal policies for Partially Observable Markov Decision
Processes (POMDPs) with Borel state, observation, and action sets
and with the expected total costs.  Action sets may not be compact
and one-step cost functions may  be unbounded.   The introduced
conditions are also sufficient for the validity of optimality
equations, semi-continuity of value functions, and convergence of
value iterations to optimal values.  Since POMDPs can be reduced
to Completely Observable Markov Decision Processes (COMDPs), whose
states are posterior state distributions, this paper focuses on
the validity of the above mentioned optimality properties for
COMDPs.  The central question is whether transition probabilities
for a COMDP are weakly continuous.  We introduce sufficient
conditions for this and show that the transition probabilities for
a COMDP are weakly continuous, if transition probabilities of the
underlying Markov Decision Process are weakly continuous and
observation probabilities for the POMDP are continuous in the
total variation. Moreover, the continuity in the total variation
of the observation probabilities cannot be weakened to setwise
continuity. The results
 are illustrated with counterexamples and examples.
\end{abstract}

\section{Introduction} \label{S1}
Partially Observable Markov Decision Processes (POMDPs) play an
important role in   operations research, electrical engineering,
and computer science. They have a broad range of applications to
various areas including sensor networks, artificial intelligence,
target tracking, control and maintenance of complex systems,
finance, and medical decision making. In principle,  it is known
how to solve POMDPs. A POMDP can be reduced to a Completely
Observable Markov Decision Process (COMDP) with the state space
being the set of belief (posterior state) probabilities for the
POMDP; see Hinderer~\cite[Section 7.1]{Hind} and Sawarigi and
Yoshikawa~\cite{SaYo} for countable state spaces and
Rhenius~\cite{Rh},
Yushkevich~\cite{Yu}, Dynkin and Yushkevich~\cite[Chapter 8]{DY},
Bertsekas and Shreve~\cite[Chapter 10]{Bert1}, and 
Hern\'{a}ndez-Lerma~\cite[Chapter 4]{HLerma} for Borel state spaces. 
After an optimal policy for the COMDP is found, it can be used to compute an optimal policy for the POMDP.  However, except
 finite state and action POMDPs (Sondik~\cite{S}),
problems with a continuous filtering transition probability $H$
(Hern\'{a}ndez-Lerma~\cite[Chapter 4]{HLerma}, Hern\'{a}ndez-Lerma
and Romera~\cite{HLerma2}), and a large variety of particular
problems considered in the literature, little is known regarding
the existence and properties of optimal policies for COMDPs and
POMDPs with  expected total costs.  

This paper investigates the existence of optimal policies for
COMDPs and therefore for POMDPs with the expected total discounted
costs and, if the one-step costs are nonnegative, with 
the expected total costs. We provide conditions for the existence of optimal policies and for the validity of other properties of optimal values and
optimal policies: they satisfy optimality equations, optimal values are lower semi-continuous functions, and value iterations converge to optimal
infinite-horizon values.

Since a COMDP is a Markov Decision Process (MDP) with Borel state
and action sets, it is natural to apply  results on the existence
of optimal policies for MDPs to COMDPs. Feinberg et al.~\cite{FKZ}
introduced a mild
assumption, called Assumption \textbf{(${\rm \bf W^*}$)}, 
for the existence of stationary optimal policies for infinite
horizon MDPs, lower semi-continuity of value functions,
characterization of the sets of optimal actions via optimality
equations, and convergence of value iterations to optimal values
for the expected total discounted  costs, if one-step costs are
bounded below, 
 and for the
expected total 
costs, if the one-step costs are nonnegative (according to the
main result in \cite{FKZ}, if another mild assumption is added to
Assumption \textbf{(${\rm \bf W^*}$)}, then there exist stationary
optimal policies for average costs per unit time). Assumption
\textbf{(${\rm \bf W^*}$)} consists  of two conditions: transition
probabilities are weakly continuous and one-step cost functions
are $\K$-inf-compact. The notion of $\K$-inf-compactness (see the
definition below) was introduced in Feinberg et al.~\cite{JMAA},
and it is slightly stronger than the lower semi-continuity of the
cost function  and its inf-compactness in the action parameter. In
operations research applications, one-step cost functions are
usually $\K$-inf-compact. Here we consider a POMDP whose
underlying MDP satisfies Assumption \textbf{(${\rm \bf W^*}$)}.
According to Theorem~\ref{th:wstar}, this implies the
$\K$-inf-compactness of the cost function for the COMDP and it
remains to prove the weak continuity of  transition probabilities
for the COMDP to verify Assumption \textbf{(${\rm \bf W^*}$)} for
the COMDP and therefore the existence of optimal policies and the
validity of additional optimality properties for the COMDP and
POMDP.

For problems with incomplete information, the filtering equation
$z_{t+1}=H(z_{t},a_t,y_{t+1})$ presented in equation (\ref{3.1})
below, that links the posterior state probabilities $z_t,
z_{t+1},$ the selected action $a_t$, and the
observation $y_{t+1},$ plays an important role. This equation 
presents a general form of Bayes's rule.
Hern\'{a}ndez-Lerma~\cite[Chapter 4]{HLerma} showed that the weak
continuity of $H$ in all three variables and the weak continuity
of transition and observation probabilities imply the weak
continuity of transition probabilities for the COMDP.   In this
paper we introduce another condition, Assumption {\bf(H)}, which
is weaker than the weak continuity of the filtering kernel $H$ in
$(z_t,a_t,y_{t+1})$. 
 We prove that this condition and setwise
continuity of the stochastic kernel on the observation set, given a posterior state 
probability and prior action, imply the weak continuity of the
transition probability for the COMDP; Theorem~\ref{th:contqqq2}.
In particular, if either these assumptions or the weak continuity
of $H$ and observation  probabilities are added to Assumption
\textbf{(${\rm \bf W^*}$)} for the underlying MDP of
the POMDP, 
the COMDP satisfies Assumption \textbf{(${\rm \bf W^*}$)} and therefore various optimality properties, including the existence of
stationary optimal policies and convergence of value iterations hold; see Theorem~\ref{main1}.

By using Assumption {\bf(H)} this paper answers the following
question: which conditions on transition and observation
probabilities are sufficient for weak continuity of transition
probabilities for the COMDP? Theorem~\ref{t:totalvar} states that
weak continuity of transition probabilities and continuity of
observation probabilities in the total variation imply weak
continuity of transition  probabilities for COMDPs. Thus,
Assumption \textbf{(${\rm \bf W^*}$)} and continuity of
observation probabilities in the total variation imply that the
COMDP satisfies Assumption \textbf{(${\rm \bf W^*}$)} and
therefore optimal policies exist for the COMDP and for the POMDP,
value iterations converge to the optimal value, and other
optimality properties hold; Theorem~\ref{main}. Example~\ref{exa1}
demonstrates that continuity of observation probabilities in the
total variation  cannot be relaxed to setwise continuity.

If the observation set is countable and it is endowed with the
discrete topology, convergence in the total variation and weak
convergence are equivalent. Thus, Theorem~\ref{t:totalvar} implies
weak continuity of the transition probability for the COMDP with a
countable observation set endowed with the discrete topology and
with weakly
continuous transition and observation kernels; Hern\'{a}ndez-Lerma~\cite[p. 93]{HLerma}. However, as 
Example~\ref{exa} demonstrates, under these conditions the
filtering transition probability $H$ may not be continuous.
This example motivated us to introduce Assumption \textbf{(H)}. 

 The main results of this paper are presented in Section~\ref{S3}.  Section~\ref{S4}
 contains three counterexamples. In addition to the two examples
 described above, Example~\ref{exa3} demonstrates that setwise
continuity of the stochastic kernel on the observation set, given a posterior state 
probability and prior action, is essential to ensure
 that Assumption {\bf(H)} implies continuity of the transition
 probability for the COMDP.  Section~\ref{S5} describes properties of
 stochastic kernels used in the proofs of main results presented in
 Section~\ref{S6}.  Section~\ref{S7} introduces a sufficient condition for
 the weak continuity of transition probabilities for the COMDP that
 combines Assumption {\bf(H)} and the weak continuity of $H$.  Combining
 these properties together is important because 
  Assumption {\bf(H)} may hold for some observations and
 continuity of $H$ may hold for others.  Section~\ref{S8} contains
 three illustrative examples:  (i) a model defined by stochastic
 equations including Kalman's filter, (ii) a model for inventory
 control with incomplete records (for particular inventory control
 problems of such  type see Bensoussan et
 al.~\cite{Bens3}--\cite{Bens5} and references therein), and (iii)  the classic Markov decision model with incomplete information studied by
 Aoki~\cite{Ao}, Dynkin~\cite{Dy}, Shiryaev~\cite{Sh},
 Hinderer~\cite[Section
 7.1]{Hind},
 Sawarigi and Yoshikawa~\cite{SaYo}, Rhenius~\cite{Rh},
Yushkevich~\cite{Yu}, Dynkin and Yushkevich~\cite[Chapter 8]{DY},
for which we provide
 a sufficient condition for the existence of optimal policies,
 convergence of value iterations to optimal values, and other
 optimality properties formulated in Theorems~\ref{main1}.

\section{Model Description}\label{S2}

For a metric space $\S$, let ${\mathcal B}(\S)$ be its Borel
$\sigma$-field, 
that is, the $\sigma$-field generated by all open subsets of the
metric space $\S$.  For a Borel set $E\in\B( \S),$ we denote by
${\mathcal B}(E)$ the $\sigma$-field whose elements are
intersections of $E$ with elements of ${\mathcal B}(\S)$.  Observe
that $E$ is a metric space with the same metric as on $\S$, and
${\mathcal B}(E)$ is its Borel $\sigma$-field. For a metric space
$\S$, we denote by $\P(\S)$ the \textit{set of probability
measures} on $(\S,{\mathcal B}(\S)).$ A sequence of probability
measures $\{\mu^{(n)}\}_{n=1,2,\ldots}$ from $\P(\S)$
\textit{converges weakly (setwise)} to $\mu\in\P(\S)$ if for any
bounded continuous (bounded Borel-measurable) function $f$ on $\S$
\[\int_\S f(s)\mu^{(n)}(ds)\to \int_\S f(s)\mu(ds) \qquad {\rm as \quad
}n\to\infty.
\]
A sequence of probability measures $\{\mu^{(n)}\}_{n=1,2,\ldots}$ from $\P(\S)$ \textit{converges in  the total  variation} to $\mu\in\P(\S)$ if
\[
\sup\left\{\int_\S f(s)\mu^{(n)}(ds)-\int_\S f(s)\mu(ds)\ : \ f:\S\to [-1,1]\mbox{ is Borel-measurable} \right\}\to 0\quad {\rm as \quad }n\to\infty.
\]
Weak convergence, setwise convergence, and convergence in the total variation are used in Y\"uksel and Linder~\cite{YL} to describe convergence of
observation channels.  Note that $\P(\S)$ is a separable metric space with respect to the weak convergence topology for probability measures, when $\S$ is
a separable metric space; Parthasarathy \cite[Chapter~II]{Part}. For separable metric spaces $\S_1$ and $\S_2$, a (Borel-measurable) \textit{stochastic
kernel} $R(ds_1|s_2)$ on $\S_1$ given $\S_2$ is a mapping $R(\,\cdot\,|\,\cdot\,):\B(\S_1)\times \S_2\to [0,1]$, such that $R(\,\cdot\,|s_2)$ is a
probability measure on $\S_1$ for any $s_2\in \S_2$, and $R(B|\,\cdot\,)$ is a Borel-measurable function on $\S_2$ for any Borel set $B\in\B(\S_1)$. A
stochastic kernel $R(ds_1|s_2)$ on $\S_1$ given $\S_2$ defines a Borel measurable mapping $s_2\to R(\,\cdot\,|s_2)$ of $\S_2$ to the metric space
$\P(\S_1)$ endowed with the topology of weak convergence. 
A stochastic kernel
$R(ds_1|s_2)$ on $\S_1$ given $\S_2$ is called 
\textit{weakly continuous (setwise continuous, continuous in  the total variation)}, if $R(\,\cdot\,|x^{(n)})$ converges weakly (setwise, in
 the total  variation) to $R(\,\cdot\,|x)$ whenever $x^{(n)}$ converges to $x$
in $\S_2$. For one-point sets $\{s_1\}\subset \S_1,$ we
sometimes write $R(s_1|s_2)$ instead of $R(\{s_1\}|s_2)$. 

For a Borel subset $S$ of a metric space $(\S,\rho)$, where $\rho$
is a metric, consider the
 metric space $(S,\rho)$.  A set $B$ is called open
(closed,  compact) in $S$ if $B\subseteq S$ and $B$ is open (closed,
compact, respectively) in $(S,\rho)$. Of course, if $S=\S$, we omit
``in $\S$''. Observe that, in general, an open (closed, compact) set
in $S$ may not be open (closed, compact, respectively).

Let $\X$, $\Y$, and $\A$ be 
Borel subsets of Polish spaces (a Polish space is a complete
separable metric space), $P(dx'|x,a)$ be a stochastic kernel on
$\X$ given $\X\times\A$, $Q(dy| a,x)$ be a stochastic kernel on
$\Y$ given $\A\times\X$, $Q_0(dy|x)$ be a stochastic kernel on
$\Y$ given $\X$, $p$ be a probability distribution on $\X$,
$c:\X\times\A\to \R=\mathbb{R}\cup\{+\infty\}$ be a bounded below
Borel function on $\X\times\A,$ where $\mathbb{R}$ is a real line.

A 
{\it POMDP} is specified by a tuple $(\X,\Y,\A,P,Q,c)$, where $\X$
is the \textit{state space}, $\Y$ is the \textit{observation set},
$\A$ is the \textit{action} \textit{set}, $P(dx'|x,a)$ is the
\textit{state transition law}, $Q(dy| a,x)$ is the
\textit{observation kernel}, $c:\X\times\A\to \R$ is the
\textit{one-step cost}.

The partially observable Markov decision process evolves as follows:
\begin{itemize}
\item at time $t=0$, the initial unobservable state $x_0$ has a
given prior distribution $p$; \item the initial observation $y_0$ is
generated according to the initial observation kernel
$Q_0(\,\cdot\,|x_0)$; \item at each time epoch $t=0,1,\ldots,$ if
the state of the system is $x_t\in\X$ and the decision-maker chooses
an action $a_t\in \A$, then the cost $c(x_t,a_t)$ is incurred; \item
the system moves to a state $x_{t+1}$ according to
the transition law $P(\,\cdot\,|x_t,a_t)$, $t=0,1,\ldots$; 
\item an observation $y_{t+1}\in\Y$ is generated by the
observation kernel $Q(\,\cdot\,|a_t,x_{t+1})$, $t=0,1,\ldots\ .$ 
\end{itemize}

Define the \textit{observable histories}: $h_0:=(p,y_0)\in \H_0$ and
$h_t:=(p,y_0,a_0,\ldots,y_{t-1}, a_{t-1}, y_t)\in\H_t$ for all
$t=1,2,\dots$,
where $\H_0:=\P(\X)\times \Y$ and $\H_t:=\H_{t-1}\times \A\times
\Y$ if $t=1,2,\dots$. A \textit{policy} $\pi$ 
for the POMDP is defined as a sequence
$\pi=\{\pi_t\}_{t=0,1,\ldots}$ 
of stochastic kernels $\pi_t$ on $\A$ given $\H_t$. A policy $\pi$ 
is called \textit{nonrandomized}, if each probability measure
$\pi_t(\,\cdot\,|h_t)$ is concentrated at one point.  The
\textit{set of all policies} is denoted by $\Pi$. The Ionescu
Tulcea theorem (Bertsekas and Shreve \cite[pp. 140-141]{Bert1} or
Hern\'andez-Lerma and Lasserre \cite[p.178]{HLerma1}) implies that
a policy $\pi\in \Pi$ and an initial distribution $p\in \P(\X)$,
together with the stochastic kernels $P$, $Q$ and $Q_0$, determine
a unique probability measure $P_{p}^\pi$ on the set of all
trajectories
$
(\X\times\Y\times \mathbb{A})^{\infty}$
endowed with the  $\sigma$-field defined by the products of Borel
$\sigma$-fields  $\B(\X)$, $\B(\Y)$, and $\B(\mathbb{A})$. 
The expectation with respect to this probability measure is denoted by $\E_{p}^\pi$.

For a finite horizon $T=0,1,...,$ 
the \textit{expected total discounted costs} are
\begin{equation}\label{eq1}
V_{T,\alpha}^{\pi}(p):=\mathbb{E}_p^{\pi}\sum\limits_{t=0}^{T-1}\alpha^tc(x_t,a_t),\qquad\qquad
p\in \P(\X),\,\pi\in\Pi,
\end{equation}
where $\alpha\ge 0$ is the discount factor,
$V_{0,\alpha}^{\pi}(p)=0.$ Consider the following assumptions. \vskip 0.9 ex 

\noindent\textbf{Assumption (D)}. $c$ is bounded below on
$\X\times\A$ and
  $\alpha\in[0,1)$.

\noindent\textbf{Assumption (P)}. $c$ is nonnegative on
$\X\times\A$ and $\alpha\in[0,1]$.\vskip 0.9 ex

When $T=\infty,$ formula (\ref{eq1}) defines the \textit{infinite
horizon expected total discounted cost}, and we denote it by
$V_\alpha^\pi(p).$ We use the notations \textbf{({\bf D})} and
\textbf{({\bf P})} following Bertsekas and Shreve~\cite[p.
214]{Bert1}, where cases~\textbf{({\bf D})}, \textbf{({\bf N})},
and \textbf{({\bf P})} are considered. However, Assumption
\textbf{({\bf D})} here is weaker than the conditions assumed in
case~\textbf{({\bf D})} in \cite[p. 214]{Bert1}, where one-step
costs are assumed to be bounded.

Since the function $c$ is bounded below on $\X\times\A$, a
discounted model can be converted into a model with nonnegative
costs by shifting the cost function. In particular, let
$c(x,a)\ge-K$ for all $(x,a)\in\X\times\A$. Consider a new cost
function $\hat{c}(x,a):=c(x,a)+K$ for all $(x,a)\in\X\times\A$. 
Then the corresponding total discounted reward is equal to
\[
\hat{V}_{\alpha}^{\pi}(p):=V_{\alpha}^{\pi}(p)+\frac{K}{1-\alpha},\qquad\qquad
\pi\in\Pi,\,p\in \P(\X).
\]
Thus, optimizing $V_{\alpha}^{\pi}$ and  $\hat{V}_{\alpha}^{\pi}$
are equivalent problems, and $\hat{V}_{\alpha}^{\pi}$ is the
objective function for a model with nonnegative costs. Though
Assumption \textbf{(P)} is  more general, Assumption \textbf{(D)}
is met in a wide range of applications. Thus we formulate the
results for either of these Assumptions. 

For any function $g^{\pi}(p)$, including $g^{\pi}(p)=V_{T,\alpha}^{\pi}(p)$ and $g^{\pi}(p)=V_{\alpha}^{\pi}(p)$, define the \textit{optimal values}
\begin{equation*} g(p):=\inf\limits_{\pi\in \Pi}g^{\pi}(p), \qquad
\ p\in\P(\X).
\end{equation*} 
A policy $\pi$ is called \textit{optimal} for the respective criterion, if $g^{\pi}(p)=g(p)$ for all $p\in \P(\X).$ For $g^\pi=V_{T,\alpha}^\pi$, the
optimal policy is called \emph{$T$-horizon discount-optimal}; for $g^\pi=V_{\alpha}^\pi$, it is called \emph{discount-optimal}.

In this paper, for the expected total costs and objective values
we use  similar notations for POMDPs, MDPs, and COMDPs. However,
we reserve the symbol $V$ for POMDPs, the symbol $v$ for MDPs, and
the notation ${\bar v}$ for COMDPs.  So, in addition to the
notations $V_{T,\alpha}^\pi,$ $V_{\alpha}^\pi,$ $V_{T,\alpha},$
and $V_{\alpha}$ introduced for POMDPs, we shall use the notations
$v_{T,\alpha}^\pi,$ $v_{\alpha}^\pi,$ $v_{T,\alpha}$, $v_{\alpha}$
and ${\bar v}_{T,\alpha}^\pi,$ ${\bar v}_{\alpha}^\pi,$ ${\bar
v}_{T,\alpha}$, ${\bar v}_{\alpha}$ for  the similar objects for
MDPs
and COMDPs, respectively. 

We recall that a function $c$ defined  on $\X\times\A$ with values in $\R$ is inf-compact 
if the set $\{(x,a)\in \X\times\A:\, c(x,a)\le \lambda\}$ is
compact  for any finite number $\lambda.$ A function $c$ defined
on $\X\times \A$ with values in $\R$ is called $\K$-inf-compact on
$\X\times\A$, if for any compact set $K\subseteq\X$, the function
$c:K\times\A\to \R$ defined on $K\times\A$ is inf-compact;
Feinberg et al.~\cite[Definition 1.1]{JMAA1, JMAA}. According to
Feinberg et al.~\cite[Lemma 2.5]{JMAA}, a bounded below function
$c$ is $\K$-inf-compact on the product of metric spaces $\X$ and
$\A$ if and only if it satisfies the following two conditions:

(a) $c$ is lower semi-continuous;

(b) if a sequence $\{x^{(n)} \}_{n=1,2,\ldots}$ with values in $\X$ converges and its limit $x$ belongs to $\X$ then any sequence $\{a^{(n)}
\}_{n=1,2,\ldots}$ with $a^{(n)}\in \A$, $n=1,2,\ldots,$ satisfying the condition that the sequence $\{c(x^{(n)},a^{(n)}) \}_{n=1,2,\ldots}$ is bounded
above, has a limit point $a\in\A.$

For a POMDP $(\X,\Y,\A,P,Q,c)$, consider the  MDP $(\X,\A,P,c)$,
in which all the states are observable. An MDP can be viewed as a
particular POMDP with $\Y=\X$ and $Q(B|a,x)=Q(B|x)={\bf I}\{x\in
B\}$ for all $x\in\X,$ $a\in \A$, and $B\in{\mathcal B}(\X)$. In
addition, for an MDP an initial state is observable.  Thus for an
MDP an initial state  $x$ is considered instead of the initial
distribution $p.$ In fact, this MDP possesses a special property
that action sets at
all the states are equal. For MDPs, Feinberg et al.~\cite{FKZ} 
provides general conditions for the existence of optimal policies,
validity of optimality equations, and convergence of value
iterations. Here we formulate these conditions for an MDP whose
action sets in all states are equal.

\noindent\textbf{Assumption (${\rm \bf W^*}$)} (cf. Feinberg et al.~\cite{FKZ} and Lemma 2.5 in \cite{JMAA}). 

(i) the function $c$ is $\K$-inf-compact on $\X\times\A$;

(ii) the transition probability  $P(\,\cdot\,|x,a)$ is weakly
continuous in $(x,a)\in \X\times\A$.

For an MDP, a nonrandomized policy is called \textit{Markov}, if
all decisions depend only on the current state and time. A Markov
policy is called \textit{stationary}, if all decisions depend only
on current states.

\begin{theorem}{\rm (cf. Feinberg et al.~\cite[Theorem~2]{FKZ}).} \label{teor4.3a}
Let MDP $(\X,\A,P,c)$ satisfy Assumption \textbf{(${\rm \bf
W^*}$)}. Let either Assumption {\rm\textbf{(P)}} or Assumption
{\rm\textbf{(D)}} hold. Then:

{(i})  the functions $v_{t,\alpha}$, $t=0,1,\ldots$, and $v_\alpha$ are lower semi-continuous on $\X$, and $v_{t,\alpha}(x)\to v_\alpha (x)$ as $t \to
\infty$ for all $x\in \X;$

{(ii)} for each $x\in \mathbb{X}$ and $t=0,1,\ldots,$
\begin{equation}\label{eq433a}
v_{t+1,\alpha}(x)=\min\limits_{a\in \A}\left\{c(x,a)+\alpha \int_\X v_{t,\alpha}(y)P(dy|x,a)\right\},
\end{equation}
where $v_{0,\alpha}(x)=0$ for all $x\in \X$, and the nonempty sets
\[
A_{t,\alpha}(x):=\left\{a\in \A\,:\, v_{t+1,\alpha}(x)=c(x,a)+\alpha \int_\X v_{t,\alpha}(y)P(dy|x,a) \right\},\quad x\in \X,\ t=0,1,\ldots,
\]
satisfy the following properties: (a) the graph ${\rm Gr}(A_{t,\alpha})=\{(x,a):\, x\in\X, a\in A_{t,\alpha}(x)\}$, $t=0,1,\ldots,$ is a Borel subset of
$\X\times \mathbb{A}$, and (b) if $v_{t+1,\alpha}(x)=+\infty$, then $A_{t,\alpha}(x)=\A$ and, if $v_{t+1,\alpha}(x)<+\infty$, then $A_{t,\alpha}(x)$ is
compact;

{(iii)} for each $T=1,2,\ldots$, there exists an optimal Markov
$T$-horizon policy $(\phi_0,\ldots,\phi_{T-1})$, and if for a
$T$-horizon Markov policy $(\phi_0,\ldots,\phi_{T-1})$ the
inclusions $\phi_{T-1-t}(x)\in A_{t,\alpha}(x)$, $x\in\X,$
$t=0,\ldots,T-1,$ hold, then this policy is $T$-horizon optimal;

{(iv)} for each $x\in \mathbb{X}$
\begin{equation}\label{eq5aa}
v_{\alpha}(x)=\min\limits_{a\in \A}\left\{c(x,a)+\alpha\int_\X
v_{\alpha}(y)P(dy|x,a)\right\},
\end{equation}
and the nonempty sets
\[
A_{\alpha}(x):=\left\{a\in \A\,:\,v_{\alpha}(x)=c(x,a)+\alpha
\int_\X v_{\alpha}(y)P(dy|x,a) \right\},\quad x\in \X,
\]
satisfy the following properties: (a) the graph ${\rm
Gr}(A_{\alpha})=\{(x,a):\, x\in\X, a\in A_\alpha(x)\}$  is a Borel
subset of $\X\times \mathbb{A}$, and (b) if $v_{\alpha}(x)=+\infty$,
then $A_{\alpha}(x)=\A$ and, if $v_{\alpha}(x)<+\infty$, then
$A_{\alpha}(x)$ is compact;

{(v)} for infinite-horizon problems there exists a stationary
discount-optimal policy $\phi_\alpha$, and a stationary policy
$\phi_\alpha^{*}$ is optimal if and only if $\phi_\alpha^{*}(x)\in
A_\alpha(x)$ for all $x\in \X$;

{(vi)}  {\rm (Feinberg and Lewis~\cite[Proposition 3.1(iv)]{FL})} if $c$ is inf-compact on $\X\times\A$, then the functions $v_{t,\alpha}$, $t=1,2,\ldots$,
and $v_\alpha$ are inf-compact on $\X$.
\end{theorem}

\section{Reduction of POMDPs to COMDPs and Main Results}\label{S3}
In this section we formulate the main results of the paper, Theorems~\ref{main1}, \ref{main},  and the relevant statements. These theorems provide
sufficient conditions for the existence of optimal policies for COMDPs and therefore for POMDPs with expected total costs, as well as optimality equations
and convergence of value iterations for COMDPs.  These conditions consist of two major components: the conditions for the existence of optimal policies for
MDPs and additional conditions on the POMDP.
 Theorem~\ref{main} states that the continuity
of the observation kernel $Q$ in the total variation is the
additional sufficient condition under which there is a stationary
optimal policy for the COMDP, and this policy satisfies the
optimality equations and can be found by value iterations. In
particular, the continuity of $Q$ in the total variation and the
weak continuity of $P$ imply the setwise continuity of the
stochastic kernel $R'$ defined in \eqref{3.5} and the validity of
Assumption {\bf {\bf(H)}} introduced in this section;
Theorem~\ref{t:totalvar}. These two additional properties imply the
weak continuity of the transition probability $q$ for the COMDP
(Theorem~\ref{th:contqqq2}) and eventually the
desired properties of the COMDP; Theorem~\ref{main1}. 

This section starts with the description of known results on the general reduction of a POMDP to the COMDP; Bertsekas and Shreve \cite[Section
10.3]{Bert1}, Dynkin and Yushkevich \cite[Chapter 8]{DY}, Hern\'{a}ndez-Lerma \cite[Chapter 4]{HLerma}, Rhenius \cite{Rh}, and Yushkevich \cite{Yu}.  To
simplify notations, we sometimes drop the time parameter.  Given a posterior distribution $z$ of the state $x$ at time epoch $t=0,1,\ldots$ and given an
action $a$ selected at epoch $t$, denote by $R(B\times C|z,a) $ the joint probability that the state at time $(t+1)$ belongs to the set $B\in {\mathcal
B}(\X)$ and the observation at time $t+1$ belongs to the set $C\in {\mathcal B}(\Y)$,
\begin{equation}\label{3.3}
R(B\times C|z,a):=\int_{\X}\int_{B}Q(C|a,x')P(dx'|x,a)z(dx),\ B\in
\mathcal{B}(\X),\ C\in \mathcal{B}(\Y),\ z\in\P(\X),\ a\in \A.
\end{equation}
Observe that $R$ is a stochastic kernel on $\X\times\Y$ given
${\P}(\X)\times \A$;  see Bertsekas and Shreve \cite[Section
10.3]{Bert1}, Dynkin and Yushkevich \cite[Chapter 8]{DY},
Hern\'{a}ndez-Lerma \cite[p. 87]{HLerma},  Yushkevich \cite{Yu} or
Rhenius \cite{Rh} for details.
 The  probability that the observation $y$ at time $t+1$
belongs to the set $C\in\B(\Y)$, given that at time $t$ the
posterior state probability is $z$ and selected action is $a,$ is
\begin{equation}\label{3.5}
R'(C|z,a):=\int_{\X}\int_{\X}Q(C|a,x')P(dx'|x,a)z(dx),\qquad
 C\in \mathcal{B}(\Y),\ z\in\P(\X),\ a\in
\A. \end{equation} 
Observe that $R'$ is a stochastic kernel on $\Y$ given
${\P}(\X)\times \A.$ By Bertsekas and Shreve~\cite[Proposition
7.27]{Bert1}, there exist a stochastic kernel $H$ on $\X$ given
${\P}(\X)\times \A\times\Y$ such that
\begin{equation}\label{3.4}
R(B\times C|z,a)=\int_{C}H(B|z,a,y)R'(dy|z,a),\quad B\in \mathcal{B}(\X),\  C\in \mathcal{B}(\Y),\ z\in\P(\X),\ a\in \A.
\end{equation}

The stochastic kernel $H(\,\cdot\,|z,a,y)$ defines a measurable
mapping $H:\,\P(\X)\times \A\times \Y \to\P(\X)$, where
$H(z,a,y)(\,\cdot\,)=H(\,\cdot\,|z,a,y).$ For each pair $(z,a)\in
\P(\X)\times\A$, the mapping $H(z,a,\cdot):\Y\to\P(\X)$ is defined
$R'(\,\cdot\,|z,a)$-almost surely  uniquely in $y\in\Y$; Bertsekas
and Shreve \cite[Corollary~7.27.1]{Bert1} or Dynkin and Yushkevich
\cite[Appendix 4.4]{DY}. 
For a posterior distribution $z_t\in \P(\X)$, action $a_t\in \A$, and an
observation $y_{t+1}\in\Y,$ the posterior distribution $z_{t+1}\in\P(\X)$ is
\begin{equation}\label{3.1}
z_{t+1}=H(z_t,a_t,y_{t+1}).
\end{equation}
However, the observation $y_{t+1}$ is not available in the COMDP model, and therefore $y_{t+1}$ is a random variable with the distribution
$R'(\,\cdot\,|z_t,a_t)$, and the right-hand side of (\ref{3.1}) maps $(z_t,a_t)\in \P(\X)\times\A$ to $\P(\P(\X)).$ Thus, $z_{t+1}$ is a random variable
with values in  $\P(\X)$ whose distribution is defined uniquely by the stochastic kernel
\begin{equation}\label{3.7}
q(D|z,a):=\int_{\Y}\h\{H(z,a,y)\in D\}R'(dy|z,a),\quad D\in \mathcal{B}(\P(\X)),\ z\in \P(\X),\ a\in\A;
\end{equation}
Hern\'andez-Lerma~\cite[p. 87]{HLerma}. The  particular choice of
a stochastic kernel $H$ satisfying (\ref{3.4}) does not effect the
definition of $q$ from (\ref{3.7}), since for each pair $(z,a)\in
\P(\X)\times\A$, the mapping $H(z,a,\cdot):\Y\to\P(\X)$ is defined
$R'(\,\cdot\,|z,a)$-almost surely uniquely in $y\in\Y$; Bertsekas
and Shreve \cite[Corollary~7.27.1]{Bert1}, Dynkin and Yushkevich
\cite[Appendix 4.4]{DY}.

Similar to the stochastic kernel $R$, consider a stochastic kernel
$R_0$ on $\X\times\Y$ given $\P(\X)$ defined by
\[
R_0(B\times C| p):=\int_{B}Q_0(C|x)p(dx), \qquad B\in \mathcal{B}(\X),\  C\in \mathcal{B}(\Y), \  p\in \P(\X).
\]
This kernel  can be decomposed as
\begin{equation}\label{3.8}
R_0(B\times C|p)=\int_{C}H_0(B|p,y)R_0'(dy|p), \quad  B\in \mathcal{B}(\X),\  C\in \mathcal{B}(\Y),\ p\in\P(\X),
\end{equation}
where $R_0'(C|p)=R_0(\X\times C|p)$, $C\in \mathcal{B}(\Y)$, $p\in
\P(\X)$,
is a stochastic kernel on $\Y$ given $\P(\X)$ and
 $H_0(dx|p,y)$ is a stochastic kernel on $\X$ given
$\P(\X)\times\Y$. Any initial prior distribution $p\in\P(\X)$ and
any initial observation $y_0$ define the initial posterior
distribution $z_0=H_0(p,y_0) $ on $(\X,\B(\X))$.  Similar to
(\ref{3.1}), the observation $y_0$ is not available in the COMDP
and this equation is stochastic.  In addition, $H_0(p,y)$ is
defined $R_0'(dy|p)$-almost surely uniquely in $y\in\Y$ for each
$p\in \P(\X).$

Similar to (\ref{3.7}), the  stochastic kernel
\begin{equation}\label{3.10}
q_0(D|p):=\int_{\Y}\h\{H_0(p,y)\in D\}R_0'(dy|p),\qquad D\in \mathcal{B}(\P(\X))\mbox{ and }p\in \P(\X),
\end{equation}
on $\P(\X)$ given $\P(\X)$ defines the the initial distribution on the set of posterior probabilities. Define $q_0(p)(D)=q_0(D|p),$ where $D\in
\mathcal{B}(\P(\X)).$ Then $q_0(p)$ is the initial distribution of $z_0=H_0(p,y_0)$ corresponding to the initial state distribution $p.$

The COMDP is defined as an MDP with parameters
($\P(\X)$,$\A$,$q$,$\c$), where
\begin{itemize}
\item[(i)] $\P(\X)$ is the state space; \item[(ii)] $\A$ is the
action set available at all states $z\in\P(\X)$; \item[(iii)] the
 one-step cost function $\c:\P(\X)\times\A\to\R$, defined as
\begin{equation}\label{eq:c}
\c(z,a):=\int_{\X}c(x,a)z(dx), \quad z\in\P(\X),\, a\in\A;
\end{equation}
 \item[(iv)] transition probabilities $q$ on $\P(\X)$
given $\P(\X)\times \A$ defined in (\ref{3.7}).
\end{itemize}
%
%
%

Denote by $i_t$, $t=0,1,\ldots$, a $t$-horizon history for the COMDP, also called an \textit{information vector},
\[
i_t:=(z_0,a_0,\ldots, z_{t-1}, a_{t-1}, z_t)\in I_t,\quad t=0,1,\ldots,
\]
where $z_0$ is the initial posterior distribution and $z_t\in \P(\X)$ are recursively defined by equation (\ref{3.1}),
$I_t:=\P(\X)\times(\A\times\P(\X))^t$ for all $t=0,1,\ldots$, with $I_0:=\P(\X)$. An \textit{information policy} ($I$-policy) is a policy in a  COMDP, i.e.
$I$-policy is a sequence $\delta=\{\delta_t\,:\,t=0,1,\dots\}$ such that $\delta_t(\,\cdot\,|i_t)$ is a stochastic kernel on $\A$ given $I_t$ for all
$t=0,1,\dots$; Bertsekas and Shreve \cite[Chapter~10]{Bert1}, Hern\'{a}ndez-Lerma \cite[p.~88]{HLerma}. Denote by $\triangle$ the set of all
\textit{I}-policies.  We also consider Markov $I$-policies and stationary $I$-policies.

For an $I$-policy $\delta=\{\delta_t\,:\,t=0,1,\dots\},$ define a policy $\pi^\delta=\{\pi^\delta_t\,:\,t=0,1,\dots\}$ in $\Pi$ as
\begin{equation}\label{eq:Per}
\pi_t^\delta(\,\cdot\,|h_t):=\delta_t(\,\cdot\,|i_t(h_t))\mbox{ for all }h_t\in H_t\mbox{ and }t=0,1,\dots,
\end{equation}
where $i_t(h_t)\in I_t$ is the information vector determined by the observable history $h_t$ via (\ref{3.1}). Thus $\delta$ and $\pi^\delta$ are equivalent
in the sense that $\pi_t^\delta$ assigns the same conditional probability on $\A$ given the observed history $h_t$ as
 $\delta_t$ for the history $i_t(h_t)$.
  If $\delta$ is an optimal policy for the COMDP then $\pi^\delta$
is an optimal policy for the POMDP.  This follows from the facts that $V_{t,\alpha}(p)={\bar v}_{t,\alpha}(q_0(p))$, $t=0,1,\ldots,$ and
$V_{\alpha}(p)={\bar v}_{\alpha}(q_0(p))$; Hern\'{a}ndez-Lerma \cite[p.~89]{HLerma} and references therein. Let $z_t(h_t)$ be the last element of the
information vector $i_t(h_t).$  With a slight abuse of notations, by using the same notations for a measure concentrated at a point and a function at this
point, if $\delta$ is Markov then \eqref{eq:Per} becomes $\pi_t^\delta(h_t)=\delta_t(z_t(h_t))$
 and if $\delta$ is stationary then
$\pi_t^\delta(h_t)=\delta(z_t(h_t)),$ $ t=0,1,\ldots\ .$

Thus, an optimal policy  for a COMDP  defines an optimal
policy for the POMDP. 
However, very little is known for the conditions on POMDPs that lead
to the existence of optimal policies for the corresponding COMDPs.
For the COMDP, Assumption \textbf{(${\rm \bf W^*}$)}
has the following form:

(i)  $\c$ is $\K$-inf-compact on $\P(\X)\times\A$;


(ii) the transition probability  $q(\,\cdot\,|z,a)$ is weakly
continuous in $(z,a)\in \P(\X)\times\A$.

Recall that the notation $\bar v$ has been reserved for the expected
total costs for COMDPs. The following theorem follows directly from
Theorem~\ref{teor4.3a} applied to the COMDP $(\P(\X),\A,q,\c)$.

\begin{theorem}
 \label{teor4.3} Let either Assumption
{\rm{\bf({\bf D})}} or Assumption {\rm\bf({\bf P})} hold. If the
COMDP $(\P(\X),\A,q,\c)$ satisfies {\rm Assumption \textbf{(${\rm
\bf W^*}$)}},    then:

{(i}) the functions ${\bar v}_{t,\alpha}$, $t=0,1,\ldots$, and
${\bar v}_\alpha$ are lower semi-continuous on $\P(\X)$, and
${\bar v}_{t,\alpha}(z)\to 
 {\bar v}_\alpha (z)$ as $t \to \infty$ for all
$z\in \P(\X);$

{(ii)} for each $z\in \P(\mathbb{X})$ and $t=0,1,...,$
\begin{equation}\label{eq433}
\begin{aligned}
&\qquad\qquad{\bar v}_{t+1,\alpha}(z)=\min\limits_{a\in \A}\left\{\c(z,a)+\alpha \int_{\P(\X)} {\bar v}_{t,\alpha}(z')q(dz'|z,a)\right\}=
\\ &\min\limits_{a\in \A}\left\{\int_{\X}c(x,a)z(dx) +\alpha \int_{\X}\int_{\X}\int_{\Y} {\bar v}_{t,\alpha}(H(z,a,y)) Q(dy|a,x')P(dx'|x,a)z(dx) \right\},
\end{aligned}
\end{equation}
where ${\bar v}_{0,\alpha}(z)=0$ for all $z\in \P(\X)$, and the
nonempty sets
\[
A_{t,\alpha}(z):=\left\{a\in \A:\,{\bar v}_{t+1,\alpha}(z)=\c(z,a)+\alpha \int_{\P(\X)} {\bar v}_{t,\alpha}(z')q(dz'|z,a) \right\},\quad z\in \P(\X),\
t=0,1,\ldots,
\]
satisfy the following properties: (a) the graph ${\rm Gr}(A_{t,\alpha})=\{(z,a):\, z\in\P(\X), a\in A_{t,\alpha}(z)\}$, $t=0,1,\ldots,$ is a Borel subset
of $\P(\X)\times \mathbb{A}$, and (b) if ${\bar v}_{t+1,\alpha}(z)=+\infty$, then $A_{t,\alpha}(z)=\A$ and, if ${\bar v}_{t+1,\alpha}(z)<+\infty$, then
$A_{t,\alpha}(z)$ is compact;

{(iii)} for each $T=1,2,\ldots$, there exists an optimal Markov
$T$-horizon $I$-policy $(\phi_0,\ldots,\phi_{T-1})$, and if for a
$T$-horizon Markov $I$-policy $(\phi_0,\ldots,\phi_{T-1})$ the
inclusions $\phi_{T-1-t}(z)\in A_{t,\alpha}(z)$, $z\in\P(\X),$
$t=0,\ldots,T-1,$ hold, then this $I$-policy is $T$-horizon
optimal;

{(iv)} for  each $z\in
\P(\X)$ 
\begin{equation}\label{eq5a}
\begin{aligned}
 &\qquad\qquad{\bar
v}_{\alpha}(z)=\min\limits_{a\in \A}\left\{\c(z,a)+\alpha\int_{\P(\X)} {\bar
v}_{\alpha}(z')q(dz'|z,a)\right\}=\\
&\min\limits_{a\in \A}\left\{\int_{\X}c(x,a)z(dx) +\alpha
\int_{\X}\int_{\X}\int_{\Y} {\bar v}_{\alpha}(H(z,a,y))
Q(dy|a,x')P(dx'|x,a)z(dx) \right\},\
\end{aligned}
\end{equation}
and the nonempty sets
\[
A_{\alpha}(z):=\left\{a\in \A:\,{\bar
v}_{\alpha}(z)=\c(z,a)+\alpha\int_{\P(\X)} {\bar
v}_{\alpha}(z')q(dz'|z,a) \right\},\quad z\in \P(\X),
\]
satisfy the following properties: (a) the graph ${\rm
Gr}(A_{\alpha})=\{(z,a):\, z\in\P(\X), a\in \A_\alpha(z)\}$ is a
Borel subset of $\P(\X)\times \mathbb{A}$, and (b) if ${\bar
v}_{\alpha}(z)=+\infty$, then $A_{\alpha}(z)=\A$ and, if ${\bar
v}_{\alpha}(z)<+\infty$, then $A_{\alpha}(z)$ is compact.

{(v)} for infinite horizon problems there exists a stationary
discount-optimal $I$-policy $\phi_\alpha$, and a stationary
$I$-policy $\phi_\alpha^{*}$ is optimal if and only if
$\phi_\alpha^{*}(z)\in A_\alpha(z)$ for all $z\in \P(\X).$

{(vi)}  if $\c$ is inf-compact on $\P(\X)\times\A$, then the functions ${\bar v}_{t,\alpha}$, $t=1,2,\ldots$, and ${\bar v}_\alpha$ are inf-compact on
$\P(\X)$.
\end{theorem}

Thus, in view of Theorem~\ref{teor4.3}, the important question is
under which conditions on the original POMDP,  the COMDP satisfies
the conditions under which there are optimal policies for MDPs.
Hern\'andez-Lerma~\cite[p. 90]{HLerma} provides the following
conditions for this: (a) $\A$ is compact, (b) the cost function
$c$ is bounded and continuous, (c) the transition probability
$P(dx'|x,a)$ and the observation kernel $Q(dy|a,x)$ are weakly
continuous stochastic kernels; (d) there exists a weakly
continuous stochastic kernel $H$ on $\X$ given
$\P(\X)\times\A\times\Y$ satisfying (\ref{3.4}). Consider the
following relaxed version of assumption (d) that does not require
that $H$ is continuous in $y$. We introduce this assumption,
called Assumption {\bf {\bf(H)}}, because it holds in many
important situations when a weakly continuous stochastic kernel
$H$ satisfying (\ref{3.4}) does not
exist; see Example~\ref{exa} 
and Theorem~\ref{t:totalvar}. 

\vspace{6pt}

\noindent\textbf{Assumption {\bf(H)}}. There exists a stochastic
kernel $H$ on $\X$ given $\P(\X)\times\A\times\Y$ satisfying
(\ref{3.4}) such that: if a sequence
$\{z^{(n)}\}_{n=1,2,\ldots}\subseteq\P(\X)$ converges weakly to
$z\in\P(\X)$, and a sequence $\{a^{(n)}\}_{n=1,2,\ldots}\subseteq\A$
converges to $a\in\A$ as $n\to\infty$, then there exists a
subsequence $\{(z^{(n_k)},a^{(n_k)})\}_{k=1,2,\ldots}\subseteq
\{(z^{(n)},a^{(n)})\}_{n=1,2,\ldots}$ and a measurable subset $C$ of
$\Y$ such that $R'(C|z,a)=1$ and for all $y\in C$
\begin{equation}\label{eq:ASSNH}
H(z^{(n_k)},a^{(n_k)},y)\mbox{ converges weakly to }H(z,a,y). 
\end{equation}
In other words, \eqref{eq:ASSNH} holds $R'(\,\cdot\,|z,a)$-almost
surely

\begin{theorem}\label{main1} If the following assumptions hold:
\begin{enumerate}[(a)]
\item  either Assumption {\rm \bf({\bf D})} or Assumption
{\rm\bf({\bf P})} holds; \item  the function $c$ is $\K$-inf-compact on $\X\times\A$;  
%
\item  either

(i) the stochastic kernel $R'(dy|z,a)$ on $\Y$ given
$\P(\X)\times\A$ is setwise continuous and Assumption {\rm\bf({\bf H})}
holds,

\noindent or

(ii) the stochastic kernels $P(dx'|x,a)$ on $\X$ given
$\X\times\A$ and $Q(dy|a,x)$ on $\Y$ given $\A\times\X$ are weakly
continuous and there exists a weakly continuous stochastic kernel
$H(dx|z,a,y)$ on $\X$ given $\P(\X)\times\A\times\Y$ satisfying
(\ref{3.4}),
\end{enumerate} then the COMDP $(\P(\X),\A,q,\c)$ satisfies
Assumption {\rm\bf(${\rm \bf W^*}$)} and therefore statements (i)--(vi)
of Theorem~\ref{teor4.3} hold.
\end{theorem}

\begin{remark}
Throughout this paper we follow the terminology according to which
finite sets are countable. If $\Y$ is countable, then equation
(\ref{eq433}) transforms into
\[
{\bar v}_{t+1,\alpha}(z)= \min\limits_{a\in
\A}\left\{\int_{\X}c(x,a)z(dx) +\alpha \sum\limits_{y\in \Y} {\bar
v}_{t,\alpha}(H(z,a,y)) R'(y|z,a) \right\},\quad z\in
\P(\mathbb{X}),\ t=0,1,...,
\]
and equation (\ref{eq5a}) transforms into
\[
{\bar v}_{\alpha}(z)= \min\limits_{a\in
\A}\left\{\int_{\X}c(x,a)z(dx) +\alpha \sum\limits_{y\in \Y} {\bar
v}_{\alpha}(H(z,a,y)) R'(y|z,a) \right\},\quad z\in
\P(\mathbb{X}).
\]
\end{remark}

Theorem~\ref{main1} follows from Theorems~\ref{teor4.3},
\ref{th:wstar}, and \ref{th:contqqq2}. In particular,
Theorem~\ref{th:wstar} implies   that if Assumption ({\bf D}) or
({\bf P}) holds for a POMDP, then it also holds  for the
corresponding
COMDP. 

%
\begin{theorem}\label{th:wstar}
If the function $c:\X\times\A\to \R$ is bounded below and
$\K$-inf-compact  on $\X\times\A$, then the cost function
$\c:\P(\X)\times\A\to\R$ defined for the COMDP in (\ref{eq:c}) is
bounded from below by the same constant as $c$ and
$\K$-inf-compact on $\P(\X)\times\A$.
\end{theorem}

\begin{theorem}\label{th:contqqq2} 
%
The stochastic kernel $q(dz'|z,a)$ on $\P(\X)$ given
$\P(\X)\times\A$ is weakly continuous if 
condition (c)
from Theorem~\ref{main1} holds. 
\end{theorem}

The following theorem provides sufficient conditions for the
existence of optimal policies for the COMDP and therefore for the
POMDP in terms of the initial parameters of the POMDP.

\begin{theorem}\label{main}
Let assumptions (a) and (b) of Theorem~\ref{main1} hold,  the
stochastic kernel $P(dx'|x,a)$ on $\X$ given $\X\times\A$ be
weakly continuous, and  the stochastic kernel $Q(dy|a,x)$ on $\Y$
given $\A\times\X$ be continuous in  the total  variation. Then
the COMDP $(\P(\X),\A,q,\c)$ satisfies {Assumption {\rm\bf(${\rm
\bf W^*}$)}} and therefore statements (i)--(vi) of
Theorem~\ref{teor4.3} hold.
\end{theorem}

Theorem~\ref{main} follows from Theorem~\ref{th:wstar} and from
the following statement.

\begin{theorem}\label{t:totalvar}
The weak continuity of the stochastic kernel $P(dx'|x,a)$ on $\X$ given $\X\times\A$ and continuity in the total variation of the stochastic kernel
$Q(dy|a,x)$ on $\Y$ given $\A\times\X$ imply that condition (i) from Theorem~\ref{main1} holds (that is, $R'$ is setwise continuous and
Assumption~{\rm\bf({\bf H})} holds) and therefore the stochastic kernel $q(dz'|z,a)$ on $\P(\X)$ given $\P(\X)\times\A$ is weakly continuous.
\end{theorem}

Example~\ref{exa1} demonstrates that, if the stochastic kernel
$Q(dy|a,x)$ on $\Y$ given $\A\times\X$ is setwise continuous, then
the transition probability $q$ for the COMDP may not be weakly
continuous in $(z,a)\in\P(\X)\times\A$. In this example the state
set consists of two points. Therefore, if the stochastic kernel
$P(dx'|x,a)$ on $\X$ given $\X\times\A$ is setwise continuous (even
if it is continuous in the total variation) in $(x,a)\in\X\times\A$
then the setwise continuity of the stochastic kernel $Q(dy|a,x)$ on
$\Y$ given $\A\times\X$  is not sufficient for the weak continuity
of $q$.

\begin{corollary}\label{cor:inf-HL123} ({\rm cp. Hern\'{a}ndez-Lerma~\cite[p.
93]{HLerma})} If the stochastic kernel $P(dx'|x,a)$ on $\X$ given
$\X\times\A$ is weakly continuous, $\Y$ is countable, and for each
$y\in \Y$ the function $Q(y|a,x)$ is  continuous on $\A\times\X$,
then  the following statements hold:

(a) for each $y\in \Y$ the function $R'(y|z,a)$ is continuous on
$\P(\X)\times\A$ with respect to the topology of weak convergence
on $\P(\X),$ and Assumption {\rm\bf({\bf H})} holds;

 (b) the
stochastic kernel $q(dz'|z,a)$ on $\P(\X)$ given $\P(\X)\times\A$
is weakly continuous;

(c) if, in addition to the above conditions, assumptions (a) and
(b) from Theorem~\ref{main1} hold, then the COMDP
$(\P(\X),\A,q,\c)$ satisfies {Assumption {\rm\bf(${\rm \bf
W^*}$)}} and therefore statements (i)--(vi) of
Theorem~\ref{teor4.3} hold.
\end{corollary}

\begin{proof}
For a    countable  $\Y$, the continuity in the total variation of
the stochastic kernel $Q(dy|x,a)$ on $\Y$ given $\A\times\X$
follows from the continuity of $Q(y|a,x)$ for each $y\in \Y$ in
$(a,x)\in \A\times\X$ and from $Q(\Y|a,x)=1$ for all $(a,x)\in
\A\times\X$. Indeed, let $(a^{(n)},x^{(n)})\to (a,x)$ as
$n\to\infty$. Since $Q(y|a,x)$ is continuous in $(a,x)\in
\A\times\X$ for each $y\in \Y$ and $Q(\Y|a,x)=1$ for all $(a,x)\in
\A\times\X$, then for any $\epsilon>0$ there exists a finite set
$Y_\epsilon\subseteq \Y$ and a natural number $N_\epsilon$ such
that $\sum_{y\in \Y\setminus Y_\epsilon} Q(\, y \,|z,a)\le
\epsilon$ and $\sum_{y\in \Y\setminus Y_\epsilon} Q(\, y
\,|z^{(n)},a^{(n)})\le 2\epsilon,$ when $n\ge N_\epsilon.$  This
and the continuity of $Q$ imply
\[
\sum_{y\in \Y} |Q(y|a^{(n)},x^{(n)})-Q(y|a,x)|\to
0\quad\mbox{as}\quad n\to\infty.
 \]
Thus, $
\sup\limits_{C\in\B(\Y)}|Q(C|a^{(n)},x^{(n)})-Q(C|a,x)|\le\sum_{y\in
\Y} |Q(y|a^{(n)},x^{(n)})-Q(y|a,x)|\to 0$ as $n\to\infty, $ and
 continuity in  the total  variation takes place. Statements (a) and (b) follow from Theorem~\ref{t:totalvar},
  and statement (c) follows from Theorem~\ref{main}.
\end{proof}

\section{Counterexamples} \label{S4}

In this section we provide three counterexamples.
Example~\ref{exa1} demonstrates that the assumption in 
Theorems~\ref{main} and \ref{t:totalvar}, that the stochastic
kernel $Q$ is continuous in the total variation,  cannot be
weakened to the assumption that $Q$ is setwise continuous.
Example~\ref{exa} shows that,  under conditions of
Corollary~\ref{cor:inf-HL123}, a weakly continuous mapping $H$
satisfying \eqref{3.1} may not exist. The existence of such a
mapping is mentioned in Hern\'{a}ndez-Lerma \cite[p.~93]{HLerma}.
Example~\ref{exa3} illustrates that  the setwise continuity of the
the stochastic kernel $R'(dy|z,a)$ on $\Y$ given $\P(\X)\times\A$
is essential in condition (i) of Theorem~\ref{main1}. Without this
assumption, Assumption {\bf {\bf(H)}} alone is not sufficient for
the weak continuity of the stochastic kernel $q(dz'|z,a)$ on
$\P(\X)$ given $\P(\X)\times\A$ and therefore for the correctness
of Theorems~\ref{main1} and \ref{th:contqqq2}.

We would like to mention
that before the authors constructed Example~\ref{exa1}, 
Huizhen Janey Yu provided them with an example when the weak
continuity of the observation kernel $Q$ is not sufficient for the
weak continuity of the stochastic probability $q(\,\cdot\,|z,a).$
In her example, $\X=\{1,2\}$, the system does not move,
$\Y=\A=[0,1]$, at state 1 the observation is 0 for any action $a$
and at state 2, under an action $a\in\A,$ the observation is
uniformly distributed on $[0,a].$ The initial belief distribution
is $z=(0.5,0.5).$

\begin{example}\label{exa1}
{\rm {\it Continuity of $Q$ in the total variation cannot be
relaxed to setwise continuity in Theorems~\ref{main} and
\ref{t:totalvar}.} Let $\X=\{1,2\}$, $\Y=[0,1]$, and
$\A=\{0\}\cup\{\frac1n :\, n=1,2,\dots\}$. The system does not
move.  This means that $P(x|x,a)=1$ for all $x=1,2$ and $a\in \A$. 
This stochastic kernel $P$ is weakly continuous and, since $\X$ is
finite, it is setwise continuous and continuous in  the total
variation.  The observation kernel $Q$ is
$Q(dy|a,1)=Q(dy|0,2)=m(dy)$, $a\in\A$, with $m$ being the Lebesgue
measure on $\Y=[0,1]$ and $Q(dy|1/n,2)=m^{(n)}(dy)$, $n=1,2,...,$
where $m^{(n)}$ is the absolutely continuous measure on $\Y=[0,1]$
with the density $f^{(n)}$,
\begin{align}\label{eqfnex1}
f^{(n)}(y)=\begin{cases} 0, &\textrm{if $2k/2^n<y<(2k+1)/2^n$ for
$k=0,1,\ldots,2^{n-1}-1$;}\\
2, &\textrm{otherwise.}
\end{cases}
\end{align}

First we show that $Q(dy|a,x)$ on $\Y$ given $\A\times\X$ is
setwise continuous in $(a,x)$. In our case, this means that the
probability distributions $Q(dy|{1}/{n},i)$ converge setwise to
$Q(dy|0,i)$ as $n\to\infty$, where $i=1,2.$  For $i=1$ this
statement is trivial, because $Q(dy|a,1)=m(dy)$ for all $a\in \A.$
For $i=2$ we need to verify that $m^{(n)}$ converge setwise to $m$
as
$n\to\infty$. According to 
Bogachev \cite[Theorem 8.10.56]{bogachev}, which is Pflanzagl's
generalization of the Fichtengolz-Dieudonn\'e-Grothendiek theorem,
measures $m^{(n)}$ converge setwise to the measure $m$, if
$m^{(n)}(C)\to m(C)$ for each open set $C$ in $[0,1]$.  Since
$m^{(n)}(0)=m(0)=m^{(n)}(1)=m(1),$ $n=1,2,\ldots,$ then
$m^{(n)}(C)\to m(C)$ for each open set $C$ in $[0,1]$ if and only
if $m^{(n)}(C)\to m(C)$ for each open set $C$ in $(0,1).$
 Choose an arbitrary open set $C$ in $(0,1).$ Then $C$ is a union of a
countable  set of open disjoint intervals $(a_i,b_i)$. Therefore, for any $\varepsilon>0$ there is a finite number $n_\varepsilon$ of open intervals
$\{(a_i,b_i):i=1,\ldots,n_\varepsilon\}$ such that $m(C\setminus C_\varepsilon)\le \varepsilon$, where $C_\varepsilon=\cup_{i=1}^{n_\varepsilon}
(a_i,b_i).$ Since $f^{(n)}\le 2$, this implies that $m^{(n)}(C\setminus C_\varepsilon)\le 2\varepsilon$ for any $n=1,2,\ldots\ .$ Since
$|m^{(n)}((a,b))-m((a,b))|\le 1/2^{n-1}, $ $n=1,2,\ldots,$ for any interval $(a,b)\subset [0,1],$  this implies that
$|m(C_\varepsilon)-m^{(n)}(C_\varepsilon)|\le \varepsilon$ if $n\ge N_\varepsilon$, where $N_\varepsilon$ is any natural number
satisfying $1/2^{N_\varepsilon-1}\le \varepsilon.$ 
Therefore, if $n\ge N_\varepsilon$ then $|m^{(n)}(C)-m(C)|\le |m^{(n)}(C_\varepsilon)-m(C_\varepsilon)|+m(C\setminus C_\varepsilon)+m^{(n)}(C\setminus
C_\varepsilon)\le 4\varepsilon.$  This implies that $m^{(n)}(C)\to m(C)$ as $n\to\infty$.  Thus $m^{(n)}$ converge setwise to $m$ as $n\to\infty$.

Second, we verify that the transition kernel $q$ does not satisfy
the weak continuity property. Consider the posterior probability
distribution $z=(z(1),z(2))=(0.5,0.5)$ of the state at the current
step. Since the system does not move, this is the prior
probability distribution at the next step. If the action $0$ is
selected at the current step then nothing new can be learned about
the state during the next step. Thus $q(z|z,0)=1$. Let $y$
be an observation at the next step, and let $D$ be the event that 
the state is 2. At the next step, the prior probability of the
event $D$ is 0.5, because $z(2)=0.5.$ Now let an action $1/n$ be
selected at the current step. The new posterior state
probabilities depend on the event $A=\{f^{(n)}(y)=2\}.$ If the
event $D$ takes place (the state is 2), then the probability of
the event $A$ is 1 and the probability of the event ${\bar
A}=\{f^{(n)}(y)=0\}$ is 0.     If the event $\bar D$ takes place
(the new state is 1), then the probabilities of the events $A$ and
$\bar A$ are 0.5. Bayes's formula implies that the posterior
probabilities   are $(1/3,2/3),$ if $f^{(n)}(y)=2,$ and $(1,0),$
if $f^{(n)}(y)=0$.  Since $f^{(n)}(2)=2$ with probability $3/4$
and $f^{(n)}(y)=0$ with probability $1/4,$ then
$q((1/3,2/3)|z,1/n)=3/4$ and $q((1,0)|z,1/n)=1/4$.  So, all the
measures $q(\,\cdot\,|z,1/n)$ are constants and they are not equal
to the measure $q(\,\cdot\,|z,0),$ which is concentrated at the
point $z=(0.5,0.5).$ Thus the transition kernel $q$ on
$\P(\X)$ given $\P(\X)\times \A$ is not weakly continuous. 
\EndPf}
\end{example}

\begin{example}\label{exa}  Under conditions of Corollary~\ref{cor:inf-HL123}
there is no weakly continuous stochastic kernel
$H(\,\cdot\,|z,a,y)$ on $\X$ given $\P(\X)\times \A\times \Y$
satisfying \eqref{3.1}. {\rm Let the state and observation spaces
 $\X=\Y=\{1,2\}$; the action space $\A=[-1,1]$; the system do not
move, that is $P(1|1,a)=P(2|2,a)=1$ for all $a\in\A;$ for each
$y\in \Y$ the observation kernel  $Q(y|a,x)$  be continuous   in
 $a\in\A$, 

\[
Q(1|a,1)=\left\{
\begin{array}{ll}
|a|,&a\in[-1,0),\\
a^2,&a\in[0,1],
\end{array}\right.\quad
Q(1|a,2)=\left\{
\begin{array}{ll}
a^2,&a\in[-1,0),\\
|a|,&a\in[0,1],
\end{array}\right.
\]
\[
Q(2|a,1)=\left\{
\begin{array}{ll}
1-|a|,&a\in[-1,0),\\
1-a^2,&a\in[0,1],
\end{array}\right.\quad
Q(2|a,2)=\left\{
\begin{array}{ll}
1-a^2,&a\in[-1,0),\\
1-|a|,&a\in[0,1];
\end{array}\right.
\]
and  $z=(z(1),z(2))=\left(\frac12,\frac12\right)$ be the
probability
measure on $\X=\{1,2\}$. 

Formula (\ref{3.3}) with $B=\{1\}$ and $C=\{1\}$ implies
\begin{equation}\label{eq:e1}
R((1,1)|z,a)=\frac12Q(1|a,1)=\left\{
\begin{array}{ll}
\frac{|a|}{2},&a\in[-1,0),\\
\frac{a^2}{2},&a\in[0,1].
\end{array}\right.
\end{equation}
Setting $C=\{1\}$ in (\ref{3.5}), we obtain
\begin{equation}\label{eq:e2}
R'(1|z,a)=\frac12 Q(1|a,1)+\frac12
Q(1|a,2)=\frac{|a|+a^2}{2}, \quad a\in[-1,1].
\end{equation}
Formulas (\ref{eq:e1}) and (\ref{eq:e2}) imply that, if $H$ satisfies (\ref{3.4}), then
\[
H(1|z,a,1)=\frac{R((1,1)|z,a)}{R'(1|z,a)}=\left\{
\begin{array}{ll}
\frac{|a|}{|a|+a^2},&a\in[-1,0),\\
\frac{a^2}{|a|+a^2},&a\in(0,1].
\end{array}\right.
\]
Therefore,
\[
\lim\limits_{a\uparrow
0}H(1|z,a,1)=1\quad\mbox{and}\quad\lim\limits_{a\downarrow
0}H(1|z,a,1)=0.
\]
Thus, the stochastic kernel $H$ on $\X$ given $\P(\X)\times \A\times \Y$ is not weakly continuous in $a$, that is, $H:\P(\X)\times \A\times \Y\to \P(\X)$
is not a continuous mapping. In view of Corollary~\ref{cor:inf-HL123},  Assumption {\rm\bf{\bf(H)}} holds.\EndPf}
\end{example}

\begin{example}\label{exa3} Stochastic kernels $P$ on $\X$ given
$\X\times\A$ and $Q$ on $\Y$ given $\A\times \X$ are weakly
continuous, the stochastic kernel $R'$ on $\Y$ given $\P(\X)\times
\A$, defined by formula (\ref{3.5}), is weakly continuous, but it
is not setwise continuous.  Though Assumption {\rm\bf{\bf(H)}}
holds, the stochastic kernel $q$ on $\P(\X)$ given
$\P(\X)\times\A$, defined by formula (\ref{3.7}), is not weakly
continuous. 

{\rm  Let $\X=\{1,2\}$,
$\Y=\A=\{1,\frac{1}{2},\frac{1}{3},\dots\}\cup \{0\}$ with the
metric $\rho(a,b)=|a-b|$, $a,b\in \Y$, and $P(x|x,a)=1$,
$x\in\X,a\in\A$. Let also $Q(0|0,x)=1$,
 $Q(0|\frac{1}{m},x)=Q(\frac{1}{n}|0,x)=0$,  $x\in\X,$ and $Q(\frac1n|\frac1m,1)=a_{m,n}\sin^2(\frac{\pi
n}{2m})$, $Q(\frac1n|\frac1m,2)=a_{m,n}\cos^2(\frac{\pi n}{2m})$,
$m,n=1,2,\ldots$\,, where $a_{m,2mk+\ell}=\frac{1}{2^{k+1}m}$ for
$k=0,1,\ldots$\,, $\ell=1,2,\ldots,\, 2m$. Since
$\sum\limits_{\ell=1}^{2m}\sin^2(\frac{\pi
\ell}{2m})=\sum\limits_{\ell=1}^{2m}\cos^2(\frac{\pi
\ell}{2m})=\sum\limits_{\ell=1}^{m}(\sin^2(\frac{\pi
\ell}{2m})+\cos^2(\frac{\pi \ell}{2m}))=m$, then
$\sum\limits_{n=1}^{\infty}Q(\frac1n|\frac1m,x)=\sum\limits_{k=0}^{\infty}\sum\limits_{\ell=1}^{2m}Q(\frac{1}{2mk+\ell}|\frac1m,x)
=\sum\limits_{k=0}^{\infty}\frac{1}{2^{k+1}}=1$, $x\in\X$, and $Q$
is a stochastic kernel on $\Y$ given $\A\times \X$.  The
stochastic kernels $P$ on $\X$ given $\X\times \A$ and $Q$ on $\Y$
given $\A\times \X$ are weakly continuous. The former is true
because of the same reasons as in Example~\ref{exa1}. The latter
is true because $\limsup_{m\to\infty}Q(C|a_m,x)\le Q(C|0,x)$ for
any closed set $C$ in $\Y.$  Indeed, a set $C$ is closed in $\Y$
if and only if either (i) $0\in C$ or (ii) $0\notin C$ and $C$ is
finite. Let $C\subseteq\Y$ be closed. In case (i),
$\limsup_{m\to\infty}Q(C|a_m,x)\le 1=Q(C|0,x)$ as $a_m\to 0$,
$x\in \X$. In case (ii), $\lim_{m\to\infty}Q(C|a_m,x)=0=Q(C|0,x)$
as $a_m\to 0$, since
$\lim_{m\to\infty}Q(\frac1n|a_m,x)=0=Q(\frac1n|0,x)$ for
$n=1,2,\ldots$ and for $x\in\X.$

Formula (\ref{3.3}) implies that
$R(1,\frac1n|z,\frac1m)=z(1)a_{m,n}\sin^2(\frac{\pi n}{2m})$,
$R(2,\frac1n|z,\frac1m)=z(2)a_{m,n}\cos^2(\frac{\pi
n}{2m})$,\newline $R(1,0|z,\frac1m)=0$, $R(2,0|z,\frac1m)=0$, and
$R(1,\frac1n|z,0)=0$, $R(2,\frac1n|z,0)=0$, $R(1,0|z,0)=z(1)$,
$R(2,0|z,0)=z(2)$ for $m,n=1,2,\ldots$, $z=(z(1),z(2))\in\P(\X)$.
Formula (\ref{3.5}) yields  $R'(0|z,\frac1m)=0$,
$R'(\frac1n|z,\frac1m)=z(1)a_{m,n}\sin^2(\frac{\pi n}{2m})+
z(2)a_{m,n}\cos^2(\frac{\pi n}{2m})$,  and $R'(0|z,0)=1$,
$R'(\frac1n|z,0)=0$ for $m,n=1,2,\ldots$\,,
$z=(z(1),z(2))\in\P(\X)$. Therefore, $R'(0|z,\frac1m)\not\to
R'(0|z,0)$ as $m\to\infty$. Thus the stochastic kernel $R'$ on
$\Y$ given $\P(\X)\times \A$ \textit{is not setwise continuous}.
However, stochastic kernel $R'$ on $\Y$ given
$\P(\X)\times \A$ is weakly continuous.

 Observe that $\P(\X)=\{(z(1),z(2)):
z(1),z(2)\ge 0,\ z(1)+z(2)=1\}\subset\mathbb{R}^2$. Let
$z=(z(1),z(2))\in\P(\X)$. If $R'(y|z,a)>0$, in view of
(\ref{3.4}), $H(x'|z,a,y)=R((x',y)|z,a)/R'(y|z,a)$ for all
$x'\in\X,$ $a\in\A$, and $y\in\Y.$    Thus, if $R'(y|z,a)>0$ then
\[
H(z,a,y)=\begin{cases} \left(\frac{z(1)\sin^2(\frac{\pi
n}{2m})}{z(1)\sin^2(\frac{\pi n}{2m})+ z(2)\cos^2(\frac{\pi
n}{2m})}, \frac{z(2)\cos^2(\frac{\pi n}{2m})}{z(1)\sin^2(\frac{\pi
n}{2m})+ z(2)\cos^2(\frac{\pi n}{2m})}\right), &{\rm if\ }
a=\frac1m,\,  y=\frac1n,\, m,n=1,2,\ldots,\\
(z(1),z(2)), &{\rm if\ } a=y=0.
 \end{cases}\]
If $R'(y|z,a)=0$, we set $H(z,a,y)= z=(z(1),z(2)).$  In particular, $H(z,\frac1m,0)= z$ for all $m=1,2,\ldots.$

Observe that \textit{Assumption {\rm\bf{\bf(H)}} holds} because, if $R'(y|z,a)>0$ and   if sequences $\{z^{(N)}\}_{N=1,2,\ldots}\subseteq\P(\X)$ and
$\{a^{(N)}\}_{N=1,2,\ldots}\subseteq\A$ converge to $z\in\P(\X)$ and $a\in\A$ respectively as $N\to\infty$, then  $H(z^{(N)},a^{(N)},y)\to H(z,a,y)$ as
$N\to\infty.$ Indeed, it is sufficient to verify this property only for the following two cases: (i) $y=\frac1n$, $a=\frac1m$, and
$R'(\frac1n|z,\frac1m)>0$, where $m,n=1,2,\ldots$, and (ii) $y=a=0.$ In case (i),   $a^{(N)}=\frac1m$, when $N$ is large enough, and
$H(z^{(N)},\frac1m,\frac1n)\to H(z,\frac1m,\frac1n)$ as $N\to\infty$ because the function $H(z,\frac1m,\frac1n)$ is continuous in $z$, when
$R'(\frac1n|z,\frac1m)>0$. For case (ii), $H(z^{(N)},a^{(N)},0)=z^{(N)}\to z$ as $N\to\infty.$

Fix $z=(\frac12,\frac12)$. According to the above formulae,
$H(z,\frac1m,\frac1n)=(\sin^2(\frac{\pi n}{2m}),\cos^2(\frac{\pi
n}{2m}))$ and \linebreak
$R'(\frac1n|z,\frac1m)=\frac{a_{m,n}}{2}$. Consider a closed
subset $D=\{(z'(1),z'(2))\in \P(\X)\,: z'(1)\ge \frac34\}$ in
$\P(\X)$. Then $q(D|z,\frac1m)=\sum\limits_{n=1,2,\ldots}
\h\{\sin^2(\frac{\pi n}{2m})\ge
\frac34\}\frac{a_{m,n}}{2}=\sum\limits_{k=0}^{\infty}\sum\limits_{\ell=1}^{2m}\h\{\sin(\frac{\pi
\ell}{2m})\ge \frac{\sqrt
3}{2}\}\frac{a_{m,2mk}+\ell}{2}=\sum\limits_{\ell=1}^{2m}\h\{\sin(\frac{\pi
\ell}{2m})\ge \frac{\sqrt 3}{2}\}\frac{1}{2m}\sum\limits_{k=
0}^{\infty}\frac{1}{2^{k+1}}\to \frac13>0$ as $m\to\infty$, where
the limit takes place because
$|[\frac{2m}{3}]-\sum\limits_{\ell=1}^{2m}\h\{\sin(\frac{\pi
\ell}{2m})\ge \frac{\sqrt 3}{2}\}|\le 1$, where $[\cdot]$ is an
integer part of a number, and $ \sum\limits_{k=
0}^\infty\frac{1}{2^{k+1}}=1$. In addition, $q(D|z,0)=0$ since
$z\notin D$ and $q(z|z,0)={\bf I}\{H(z,0,0)=z\}R'(0|z,0)=1$. Thus,
$\lim\limits_{m\to\infty} q(D|z,\frac1m)=\frac13 >0= q(D|z,0)$ for
a closed set $D$ in $\P(\X).$ This implies that the stochastic
kernel $q$ on $\P(\X)$ given $\P(\X)\times\A$ \emph{is not weakly
continuous}. \EndPf
 }\end{example}

\section{Continuity of Transition Kernels for Posterior
Probabilities}\label{S5}

This section contains the proofs of Theorems~\ref{th:contqqq2} and
\ref{t:totalvar}. The following two versions of Fatou's lemma for a
sequence of measures $\{\mu^{(n)}\}_{n=1,2,\ldots}$ are used in the
proofs provided
below. 

\begin{lemma}\label{lem:F} {\rm (Generalized Fatou's Lemma).}
Let $\S$ be an arbitrary metric space, $\{\mu^{(n)}\}_{n=1,2,\ldots} \subset \mathbb{P}(\S)$, and $\{f^{(n)}\}_{n=1,2,\ldots}$ be a sequence of measurable
nonnegative $\overline{\mathbb{R}}$-valued functions on $\S$. Then:
\begin{itemize}
\item[(i)] {\rm(Royden \cite[p. 231]{Ro})} if $\{\mu^{(n)}\}_{n=1,2,\ldots}
\subset \mathbb{P}(\S)$ converges  setwise to $\mu\in
\mathbb{P}(\S)$, then
\begin{equation}\label{eq3aaa}\int_\S \liminf_{n\to\infty}f^{(n)}(s)\mu(ds)\le \liminf_{n\to\infty}\int_\S
f^{(n)}(s)\mu^{(n)}(ds);
\end{equation}
\item[(ii)] {\rm(Sch\"al~ \cite[Lemma 2.3(ii)]{Schal}, Jaskiewicz
and Nowak~\cite[Lemma~3.2]{in}, Feinberg et al.~\cite[Lemma 4]{FKZ}, \cite[Theorem~1.1]{TVP})} if $\{\mu^{(n)}\}_{n=1,2,\ldots} \subset \mathbb{P}(\S)$
converges weakly to $\mu\in \mathbb{P}(\S)$, then
\begin{equation}\label{eq3.1aaa}
\int_\S \ilim\limits_{n\to\infty,\, s'\to s}f^{(n)}(s')\mu(ds)\le \ilim\limits_{n\to \infty}\int_\S f^{(n)}(s)\mu^{(n)}(ds).
\end{equation}
\end{itemize}
\end{lemma}

\begin{proof}[Proof of Theorem~\ref{th:contqqq2}]
According to Parthasarathy \cite[Theorem~6.1, p.~40]{Part}, Shiryaev \cite[p.~311]{Sh}, Billingsley \cite[Theorem 2.1]{Bil}, the stochastic kernel
$q(dz'|z,a)$ on $\P(\X)$ given $\P(\X)\times \A$ is weakly continuous if and only if $q(D|z,a)$ is lower semi-continuous in $(z,a)\in (\P(\X)\times\X)$ for
every open set
$D$ in $\P(\X),$ that is, 
\begin{equation}\label{eq:q1222111}
\ilim_{n\to\infty}q(D|z^{(n)},a^{(n)})\ge q(D|z,a),
\end{equation}
for all $z,z^{(n)}\in\P(\X)$, and $a,a^{(n)}\in\A$, $n=1,2,\dots$, such that $z^{(n)}\to z$   weakly and
  $a^{(n)}\to a.$ 

  To prove (\ref{eq:q1222111}), suppose that
\[
\ilim_{n\to\infty}q(D|z^{(n)},a^{(n)})< q(D|z,a).
\]
Then there exists $\varepsilon^*>0$ and a subsequence $\{ z^{(n,1)},a^{(n,1)}\}_{n=1,2,\ldots}\subseteq\{z^{(n)},a^{(n)}\}_{n=1,2,\ldots}$  such that
\begin{equation}\label{eq:q2111}
q(D|z^{(n,1)},a^{(n,1)})\le q(D|z,a)-\varepsilon^*,\qquad\qquad  n=1,2,\ldots\ .
\end{equation}

If condition (ii) of Theorem~\ref{main1} holds, 
then formula (\ref{3.7}), the weak continuity of the stochastic kernel $R'$ on $\Y$ given $\P(\X)\times\A$ (this weak continuity is
proved in Hern\'{a}ndez-Lerma 
\cite[p.~92]{HLerma}), and Lemma~\ref{lem:F}(ii) 
contradict (\ref{eq:q2111}). If  condition (i) of
Theorem~\ref{main1} holds, 
then  there exists a subsequence
$\{z^{(n,2)},a^{(n,2)}\}_{n=1,2,\ldots}$ $\subseteq
\{z^{(n,1)},a^{(n,1)} \}_{n=1,2,\ldots}$ such that $
H(z^{(n,2)},a^{(n,2)},y)\to H(z,a,y)$ weakly as $n\to \infty,$
$R'(\,\cdot\,|z,a)$-almost surely in $y\in\Y.$
Therefore, since $D$ is an open set in $\P(\X)$,
\begin{equation}\label{eq:q0000}
\ilim_{n\to\infty}\h\{H(z^{(n,2)},a^{(n,2)},y)\in D\}\ge
\h\{H(z,a,y)\in D\},\quad R'(\,\cdot\,|z,a)\mbox{-almost surely in
}y\in\Y.
\end{equation}
Formulas (\ref{3.7}), (\ref{eq:q0000}), the setwise continuity of
the stochastic kernel $R'$ on $\Y$ given $\P(\X)\times\A$, and
Lemma~\ref{lem:F}(i) imply $
\ilim_{n\to\infty}q(D|z^{(n,2)},a^{(n,2)})\ge q(D|z,a), $ which
contradicts (\ref{eq:q2111}). Thus (\ref{eq:q1222111})
holds. 
\end{proof}

In order to prove Theorem~\ref{t:totalvar}, we need to formulate
and prove several auxiliary facts. 
Let $\S$ be a metric space, $\F(\S)$ and $\C(\S)$ be respectively
the spaces of all real-valued functions and all bounded continuous
functions defined on $\S$. A subset $\mathcal{A}_0\subseteq \F(\S)$
is said to be \textit{equicontinuous at a point $s\in\S$}, if $
\sup\limits_{f\in\mathcal{A}_0}|f(s')-f(s)|\to 0$ as $s'\to s. $ A
subset $\mathcal{A}_0\subseteq \F(\S)$ is said to be
\textit{uniformly bounded}, if there exists a constant constant
$M<+\infty $ such that $ |f(s)|\le M,$ for all $s\in\S$ and  for all
$f\in\mathcal{A}_0. $  Obviously, if a subset
$\mathcal{A}_0\subseteq \F(\S)$ is equicontinuous at all the points
$s\in\S$ and uniformly bounded, then $\mathcal{A}_0\subseteq
\C(\S).$

%
%

\begin{theorem}\label{kern}
Let $\S_1$, $\S_2$, and $\S_3$ be arbitrary metric spaces,
$\Psi(ds_2|s_1)$ be a weakly continuous stochastic kernel on $\S_2$
given $\S_1$, and a subset $\mathcal{A}_0\subseteq
\C(\S_2\times\S_3)$ be equicontinuous at all the points
$(s_2,s_3)\in\S_2\times\S_3$ and uniformly bounded. If $\S_2$ is
separable, then for every open set $\oo$ in $\S_2$ the family of
functions defined on $\S_1\times\S_3$,
\[
\mathcal{A}_\oo=\left\{(s_1,s_3)\to\int_{\oo}f(s_2,s_3)\Psi(ds_2|s_1)\,:\, f\in\mathcal{A}_0\right\},
\]
is equicontinuous at all the points $(s_1,s_3)\in\S_1\times\S_3$ and uniformly bounded.
\end{theorem}

\begin{proof}  The family $\mathcal{A}_\emptyset$ consists of a
single function, which is identically equal to 0.  Thus, the
statement of the theorem holds when $\oo=\emptyset.$  Let $\oo\ne
\emptyset.$ Since $\mathcal{A}_0\subseteq \C(\S_2\times\S_3)$ is
uniformly bounded, then
\begin{equation}\label{eq:kern123}
\mathcal{M}=\sup\limits_{f\in\mathcal{A}_0 }\sup\limits_{s_2\in\S_2}\sup\limits_{s_3\in\S_3}|f(s_2,s_3)|<\infty,
\end{equation}
and, since $\Psi(ds_2|s_1)$ is a stochastic kernel, the family of
functions $\mathcal{A}_\oo$ is uniformly bounded by $\mathcal{M}$.  

Let us fix an arbitrary nonempty open set $\oo\in\S_2$ and an
arbitrary point $(s_1,s_3)\in\S_1\times\S_3$. We shall prove that
$\mathcal{A}_\oo\subset \F(\S_1\times\S_3)$ is equicontinuous at the
point $(s_1,s_3)$. For any $s\in \S_2$ and $\delta>0$ denote by
$B_\delta(s)$ and $\bar{B}_\delta(s)$ respectively the open and
closed balls in the
metric space $\S_2$ of the radius $\delta$ with the center $s$ 
and by $S_\delta(s)$ the sphere in 
$\S_2$ of the radius $\delta$ with the center $s$. Note that
$S_\delta(s)=\bar{B}_\delta(s)\setminus B_\delta(s)$ is the
boundary of $B_\delta(s)$. Every ball $B_\delta(s)$ contains a
ball $B_{\delta'}(s)$, $0<\delta'\le\delta$, such that
\[
\Psi(\bar{B}_{\delta'}(s)\setminus B_{\delta'}(s)|s_1)=\Psi(S_{\delta'}(s)|s_1)=0,
\]
that is, $B_{\delta'}(s)$ is a continuity set for the probability 
measure $\Psi(\,\cdot\,|s_1)$; Parthasarathy \cite[p.~50]{Part}. Since $\oo$ is an open set in $\S_2$, for any $s\in \oo$ there exists $\delta_s>0$ such
that $B_{\delta_s}(s)$ is a continuity set for a probability measure $\Psi(\,\cdot\,|s_1)$ and $B_{\delta_s}(s)\subseteq\oo$. The family
$\{B_{\delta_s}(s)\,:\,
s\in \oo\}$ is a cover of $\oo$. Since $\S_2$ is a separable metric 
space, by Lindel\"of's lemma, there exists a sequence $\{s^{(j)}\}_{j=1,2,\ldots}\subset \oo$ such
that $\{B_{\delta_{s^{(j)}}}(s^{(j)})\,:\, j=1,2,\dots \}$ is a cover of 
$\oo$. The sets
\[
A^{(1)}:=B_{\delta_{s^{(1)}}}(s^{(1)}),\ A^{(2)}:=B_{\delta_{s^{(2)}}}(s^{(2)})\setminus B_{\delta_{s^{(1)}}}(s^{(1)}),\ \dots,\
A^{(j)}:=B_{\delta_{s^{(j)}}}(s^{(j)})\setminus\left(\cup_{i=1}^{j-1} B_{\delta_{s^{(i)}}}(s^{(i)})\right),\dots
\]
are continuity sets for the probability measure
$\Psi(\,\cdot\,|s_1)$. In view of Parthasarathy
\cite[Theorem~6.1(e), p.~40]{Part}, 
\begin{equation}\label{eq:kern1b}
\Psi(A^{(j)}|s_1')\to \Psi(A^{(j)}|s_1) \quad{\rm as}\  s_1'\to s_1,\quad j=1,2,\dots\ .
\end{equation}
Moreover,
\begin{equation}\label{eq:kern1a}
\cup_{j=1,2,\ldots}A^{(j)}=\oo\quad\mbox{and}\quad A^{(i)}\cap A^{(j)}=\emptyset \quad{\rm for\  all}\quad  i\ne j.
\end{equation}

The next step of the proof is to show that for each $j=1,2,\dots$
\begin{equation}\label{eq:kern2}
\sup\limits_{f\in\mathcal{A}_0 }\left|\int_{A^{(j)}} f(s_2,s_3')\Psi(ds_2|s_1')-\int_{A^{(j)}} f(s_2,s_3)\Psi(ds_2|s_1)\right|\to 0 \qquad {\rm as}\
(s_1',s_3')\to (s_1,s_3).
\end{equation}

Fix an arbitrary $j=1,2,\dots$\ . If $\Psi(A^{(j)}|s_1)=0$, then
formula (\ref{eq:kern2}) directly follows from (\ref{eq:kern1b})
and (\ref{eq:kern123}). Now let $\Psi(A^{(j)}|s_1)>0$. Formula
(\ref{eq:kern1b}) implies the existence of such $\delta>0$ that
$\Psi(A^{(j)}|s_1')>0$ for all $s_1'\in B_\delta(s_1)$. We endow
$A^{(j)}$ with the induced topology from $\S_2$ and set
\[
\Psi_j(C|s_1'):=\frac{\Psi(C|s_1')}{\Psi(A^{(j)}|s_1')},\quad s_1'\in B_\delta(s_1), \ C\in \B(A^{(j)}).
\]
Formula (\ref{eq:kern1b}) yields
\begin{equation}\label{eq:kern3a}
\Psi_j(ds_2|s_1')\mbox{ converges weakly to } \Psi_j(ds_2|s_1)\mbox{ in } \P(A^{(j)})\  {\rm as}\  s_1'\to s_1.
\end{equation}
%
%
According to Parthasarathy \cite[Theorem~6.8, p.~51]{Part},
\begin{equation}\label{eq:kern123*}
\sup\limits_{f\in\mathcal{A}_0 }\left|\int_{A^{(j)}} f(s_2,s_3)\Psi(ds_2|s_1')-\int_{A^{(j)}} f(s_2,s_3)\Psi(ds_2|s_1)\right|\to 0 \qquad {\rm as}\ s_1'\to
s_1.
\end{equation}
Equicontinuity of $\mathcal{A}_0\subseteq \C(\S_2\times\S_3)$ at
all the points $(s_2,s_3)\in\S_2\times\S_3$ and the inequality $
|f(s_2',s_3')-f(s_2',s_3)|\le
|f(s_2',s_3')-f(s_2,s_3)|+|f(s'_2,s_3)-f(s_2,s_3)|$ imply
\begin{equation}\label{eq:kern123**}
\slim\limits_{(s_2',s_3')\to(s_2,s_3)}\sup\limits_{f\in\mathcal{A}_0
}|f(s_2',s_3')-f(s_2',s_3)|=0\quad{\rm  for\ all\quad }
s_2\in\S_2.
\end{equation}
Thus, formulas (\ref{eq:kern123**}), (\ref{eq:kern3a}) and
Lemma~\ref{lem:F}(ii) imply 
\begin{equation}\label{eq:kern123***}\begin{aligned}
&\slim\limits_{(s_1',s_3')\to(s_1,s_3)} \sup\limits_{f\in\mathcal{A}_0 }\left|\int_{A^{(j)}} \left(f(s_2,s_3')-f(s_2,s_3)\right)\Psi(ds_2|s_1')\right|\\
%
%
&\le \int_{A^{(j)}}\slim\limits_{(s_2',s_3')\to(s_2,s_3)}\sup\limits_{f\in\mathcal{A}_0 }\left|f(s_2',s_3')-f(s_2',s_3)\right|\Psi(ds_2|s_1)=0.
\end{aligned}\end{equation} Formula (\ref{eq:kern2}) follows from
(\ref{eq:kern123*}) and (\ref{eq:kern123***}).

Since, for all $j=1,2,\dots$ and for all $(s_1',s_3')\in
\S_1\times\S_3$,
\[
\sup\limits_{f\in\mathcal{A}_0 }\left|\int_{A^{(j)}} f(s_2,s_3')\Psi(ds_2|s_1')-\int_{A^{(j)}} f(s_2,s_3)\Psi(ds_2|s_1)\right|\le
2\mathcal{M}\Psi(A^{(j)}|s_1),
\]
and $\sum\limits_{j=1}^{\infty} \Psi(A^{(j)}|s_1)= \Psi(\oo|s_1)\le1$, then equicontinuity of $\mathcal{A}_\oo$ at the point $(s_1,s_3)$ follows from
(\ref{eq:kern1a}) and (\ref{eq:kern2}). Indeed,
\[
\sup\limits_{f\in\mathcal{A}_0 }\left|\int_{\oo} f(s_2,s_3')\Psi(ds_2|s_1')-\int_{\oo} f(s_2,s_3)\Psi(ds_2|s_1)\right|
\]
\[
\le \sum\limits_{j=1,2,\dots } \sup\limits_{f\in\mathcal{A}_0 }\left|\int_{A^{(j)}} f(s_2,s_3')\Psi(ds_2|s_1')-\int_{A^{(j)}}
f(s_2,s_3)\Psi(ds_2|s_1)\right|\to0\quad {\rm as}\ (s_1',s_3')\to (s_1,s_3).
\]

As $(s_1,s_3)\in\S_1\times\S_3$ is arbitrary,
the above inequality implies that $\mathcal{A}_\oo$ is
equicontinuous at all the points $(s_1,s_3)\in \S_1\times\S_3$. In
particular, $\mathcal{A}_\oo\subseteq \C(\S_1\times\S_3)$.
\end{proof}

For a set $B\in\B(\X)$, let $\mathcal{R}_B$ be the following
family of functions defined on $\P(\X)\times \A$: 
\begin{equation}
\label{R-def} \mathcal{R}_B=\left\{(z,a)\to R( B\times C|z,a)\,:\, C\in\B(\Y)\right\}.
\end{equation}

\begin{lemma}\label{corRsetmin}
Let the stochastic kernel $P(dx'|x,a)$ on $\X$ given $\X\times\A$ be
weakly continuous and the stochastic kernel $Q(dy|a,x)$ on $\Y$
given $\A\times\X$ be continuous in the
 total  variation.  Consider the stochastic kernel $R(\,\cdot\,|z,a)$ on $\X\times\Y$ given $\P(\X)\times\A$
 defined in formula (\ref{3.3}). Then, for every pair of 
 open sets $\oo_1$ and $\oo_2$ in $\X$, the family of functions $\mathcal{R}_{\oo_1 \setminus \oo_2}$
defined on $\P(\X)\times \A$ is uniformly bounded and is equicontinuous at all the points $(z,a)\in \P(\X)\times\A$, that is, for all $z, z^{(n)} \in
\P(\X)$, $a, a^{(n)} \in \A$, $n = 1,2,\ldots$, such that $z^{(n)} \to z$ weakly and $a^{(n)} \to a$,
\begin{equation}
\label{eq:EC1} \sup\limits_{C\in\B(\Y)} |R((\oo_1\setminus\oo_2) \times C|z^{(n)}, a^{(n)}) - R((\oo_1\setminus\oo_2)\times C|z,a)| \to 0.
\end{equation}
\end{lemma}
\begin{proof}  Since $R$ is a
stochastic kernel, all the functions in the family $\mathcal{R}_{\oo_1 \setminus \oo_2}$ are nonnegative and bounded above by 1.
Thus, this family is
uniformly bounded. The remaining proof establishes the equicontinuity of $\mathcal{R}_{\oo_1 \setminus \oo_2}$ at all
the points $(z,a)\in \P(\X)\times\A$.
First we show that 
$\mathcal{R}_{\oo}$ is equicontinuous at all the points $(z,a)$ when
$\oo$ is an open set in $\X$. 
Theorem~\ref{kern}, with $\S_1=\X\times\A$, $\S_2=\X$, $\S_3=\A$, $\oo=\oo$,  $\Psi=P,$ and 
$\mathcal{A}_0=\left\{(a,x)\to Q(C|a,x)\,:\,
C\in\B(\Y)\right\}\subseteq \C(\A\times\X)$, implies that the
 the family of functions
$\mathcal{A}_{\oo}^1=\left\{(x,a)\to\int_{\oo}Q(C|a,x')P(dx'|x,a)\,:\,C\in\B(\Y)\right\}$
is equicontinuous at all the points 
$(x,a)\in \X\times\A$. In particular, $\mathcal{A}_{\oo}^1\subseteq
 \C(\A\times\X)$. Thus, Theorem~\ref{kern}, with $\S_1=\P(\X)$, $\S_2=\X$, $\S_3=\A$, $\oo=\X$, $\Psi(B|z)=z(B)$, $B\in\B(\X)$, $z\in\P(\X)$, and
$\mathcal{A}_0=\mathcal{A}_{\oo}^1$, implies that the family
$\mathcal{R}_{\oo}$ is equicontinuous at all the points
$(z,a) \in\P(\X)\times\A$.  
Second, let $\oo_1$ and $\oo_2$ be arbitrary open sets in $\X$. Then the families of functions $\mathcal{R}_{\oo_1}$, $\mathcal{R}_{\oo_2}$, and
$\mathcal{R}_{\oo_1 \cup \oo_2}$ are equicontinuous at all the points $(z, a)\in \P(\X)\times\A$. Thus, for all $z, z^{(n)} \in \P(\X)$, $a, a^{(n)} \in
\A$, $n = 1,2,\ldots$, such that $z^{(n)} \to z$ weakly and $a^{(n)} \to a$,
\begin{equation}
\begin{aligned}
\label{eq:EC}
&\sup\limits_{C\in\B(\Y)} |R((\oo_1\setminus\oo_2)  \times C|z^{(n)}, a^{(n)}) - R((\oo_1\setminus\oo_2)\times C|z,a)| \\
&\le \sup\limits_{C\in\B(\Y)}|R((\oo_1\cup\oo_2)\times C|z^{(n)},a^{(n)})-R((\oo_1\cup\oo_2)\times C|z,a)|\\
&+\sup\limits_{C\in\B(\Y)}|R(\oo_2\times C|z^{(n)},a^{(n)})-R(\oo_2\times C|z,a)|\to 0,
\end{aligned}
\end{equation}
that is, the family of functions
$\mathcal{R}_{\oo_1\setminus\oo_2}$ 
is equicontinuous at all the points $(z,a)\in \P(\X)\times\A$.
\end{proof}

\begin{corollary}\label{corR'}
Let assumptions of Lemma~\ref{corRsetmin} hold. Then the stochastic
kernel $R'(dy|z,a)$ on $\Y$ given $\P(\X)\times\A$, defined in
formula (\ref{3.5}), is continuous in  the total  variation.
\end{corollary}
\begin{proof}
This corollary follows from  Lemma~\ref{corRsetmin} applied to
$\oo_1=\X$ and $\oo_2 = \emptyset$.
\end{proof}

\begin{theorem}\label{setwise}
Let $\S$ be an arbitrary metric space, $\{h,\, h^{(n)}\}_{n=1,2,\ldots}$ be 
Borel-measurable uniformly bounded real-valued functions on $\S$, $\{\mu^{(n)}\}_{n=1,2,\ldots}\subset \P(\S)$ converges in  the total  variation to
$\mu\in\P(\S)$, and
\begin{equation}\label{sw1}
\sup\limits_{S\in\B(\S)}\left|\int_{S}h^{(n)}(s)\mu^{(n)}(ds)- \int_{S}h(s)\mu(ds)\right|\to 0 \quad {\rm as}\quad n\to\infty.
\end{equation}
Then $\{h^{(n)}\}_{n=1,2,\ldots}$ converges in probability $\mu $ to  $h$, and therefore there is a subsequence $\{n_k\}_{k=1,2,\ldots}$ such that
$\{h^{(n_k)}\}_{k=1,2,\ldots}$ converges $\mu$-almost surely to $h.$ 
\end{theorem}
\begin{proof}
Fix an arbitrary $\varepsilon
>0$ and set
\[
S^{(n,+)}:=\{s\in\S\,:\, h^{(n)}(s)-h(s)\ge \varepsilon\}, \ S^{(n,-)}:=\{s\in\S\,:\, h(s)-h^{(n)}(s)\ge \varepsilon\},
\]
\[
S^{(n)}:=\{s\in\S\,:\, |h^{(n)}(s)-h(s)|\ge \varepsilon\}=S^{(n,+)}\cup S^{(n,-)}, \quad n=1,2,\ldots\ .
\]
Note that for all $n=1,2,\ldots$
\begin{equation}\label{sw5}\begin{aligned}
\varepsilon \mu^{(n)}(S^{(n,+)})\le\int_{S^{(n,+)}}
h^{(n)}(s)\mu^{(n)}(ds) & -\int_{S^{(n,+)}} h(s)\mu^{(n)}(ds)\\
\le\left|\,\int_{S^{(n,+)}} h^{(n)}(s)\mu^{(n)}(ds)-\int_{S^{(n,+)}} h(s)\mu(ds)\right|&+ \left|\,\int_{S^{(n,+)}} h(s)\mu^{(n)}(ds) -\int_{S^{(n,+)}}
h(s)\mu(ds)\right|
 .\end{aligned}
\end{equation}
The convergence in the total variation of $\mu^{(n)}$ to $\mu\in \P(\S)$  implies that
\begin{equation}\label{sw5a}
\left|\,\int_{S^{(n,+)}} h(s)\mu^{(n)}(ds)-\int_{S^{(n,+)}} h(s)\mu(ds)\right|\to 0 {\rm\ and\ } |\mu^{(n)}(S^{(n,+)})-\mu(S^{(n,+)})|\to 0 {\rm\ as\ }
n\to\infty.
\end{equation}
Formulas (\ref{sw1})--(\ref{sw5a}) yield
$\left|\int_{S^{(n,+)}} h^{(n)}(s)\mu^{(n)}(ds)-\int_{S^{(n,+)}} h(s)\mu(ds)\right|\to 0$ as $n\to\infty$, and, in view of (\ref{sw5}),
\begin{equation}\label{sw6}
\mu^{(n)}(S^{(n,+)})\to 0\quad {\rm as} \quad n\to\infty.
\end{equation}
Being applied to the functions $\{-h,-h^{(n)}\}_{n=1,2,\ldots}$,
formula (\ref{sw6}) implies that $\mu^{(n)}(S^{(n,-)})\to 0$ as
$n\to\infty.$ This fact, (\ref{sw6}) and
%
%
%
the convergence in  the total  variation of $\mu^{(n)}$ to $\mu$
in $\P(\S)$ imply
\begin{equation*}\begin{aligned}
\mu(S^{(n)})=\mu(S^{(n,+)})&+\mu(S^{(n,-)})\\
\le |\mu(S^{(n,+)})-\mu^{(n)}(S^{(n,+)})|+ |\mu(S^{(n,-)})-\mu^{(n)}(S^{(n,-)})|&+\mu^{(n)}(S^{(n,+)})+\mu^{(n)}(S^{(n,-)}) \to0 {\rm\ as\ }  n\to\infty.
\end{aligned}\end{equation*}
Since $\varepsilon>0$ is arbitrary, $\{h^{(n)}\}_{n=1,2,\ldots}$
converges to $h$ in probability $\mu$ and, therefore,
$\{h^{(n)}\}_{n=1,2,\ldots}$ contains a subsequence
$\{h^{(n_k)}\}_{k=1,2,\ldots}$ that converges $\mu$-almost surely to
$h.$
\end{proof}

\begin{lemma}
\label{R-H} If the topology on $\X$ has a countable base $\tau_b=\{\oo^{(j)}\}_{j=1,2,\ldots}$ such that, for each finite intersection $\oo=\cap_{i=1}^ N
{\oo}^{(j_i)}$ of its elements $\oo^{(j_i)}\in\tau_b,$ $i=1,2,\ldots,N$,  the family of functions
$\mathcal{R}_{\oo}$ defined in \eqref{R-def} 
is equicontinuous at all the points $(z,a) \in \P(\X) \times \A$,
then for any sequence $\{(z^{(n)}, a^{(n)})\}_{n =1,2,\ldots},$
such that $\{z^{(n)}\}_{n =1,2,\ldots} \subseteq \P(\X)$ converges
weakly to $z \in \P(\X)$ and $\{a^{(n)}\}_{n =1,2,\ldots}
\subseteq \A $ converges to
$a$, there exists a subsequence $\{(z^{(n_k)}, a^{(n_k)})\}_{k =1,2,\ldots}$ and a set $C^* \in \B(\Y)$ such that  
\begin{equation}
\label{result} R'(C^* | z,a)  = 1 \mbox{ and } H(\, \cdot\, | z^{(n_k)}, a^{(n_k)},y ) \mbox{ converges weakly to }  H(\, \cdot\, | z,a,y) \mbox{ for all
} y \in C^*,
\end{equation}
and, therefore, Assumption {\rm\bf ({H})} holds.
\end{lemma}
As clear from the proof of Lemma~\ref{R-H}, 
the intersection assumption is equivalent to the similar
assumption for finite unions. However, in this paper we use the
intersection assumption.  
\begin{proof}
According to 
 Billingsley \cite[Theorem~2.1]{Bil} or Shiryaev \cite[p.~311]{Sh},
\eqref{result} holds if  there exists a subsequence
$\{(z^{(n_k)}, a^{(n_k)})\}_{k =1,2,\ldots}$ of the sequence $\{(z^{(n)}, a^{(n)})\}_{n =1,2,\ldots}$ and a set $C^* \in \B(\Y)$ such that 
\begin{equation}
\label{eq:q1EF}
R'(C^*|z,a)=1\quad{\rm and} \quad \ilim\limits_{k\to\infty}H({\mathcal O}|z^{(n_k)},a^{(n_k)},y)\ge H({\mathcal O}|z,a,y)\ {\rm for\ all}\ y\in C^*,
\end{equation}
for all open sets $\mathcal{O}$ in $\X$.  The rest of the proof establishes the existence of a subsequence $\{(z^{(n_k)}, a^{(n_k)})\}_{k =1,2,\ldots}$ of
the sequence $\{(z^{(n)}, a^{(n)})\}_{n =1,2,\ldots}$ and a set $C^* \in \B(\Y)$ such that \eqref{eq:q1EF} holds for all open sets $\mathcal{O}$ in $\X$.

Let $\mathcal{A}_1$ be a family of all the subsets of $\X$ that
are
 finite unions of sets from $\tau_b$, and let
$\mathcal{A}_2$ be a family of all subsets $B$ of $\X$ such that $B = \mathcal{\tilde{O}} \setminus \mathcal{O} '$ with $\mathcal{\tilde{O}}  \in \tau_b$
and $\mathcal{O} ' \in \mathcal{A}_1$. Observe that: (i) both $\mathcal{A}_1$ and
$\mathcal{A}_2$ are countable, 
(ii) any open set $\oo$ in $\X$ can be represented as
\begin{equation}
\label{eq:equi} \oo = \bigcup_{j =1,2,\ldots} \oo^{(j,1)} = \bigcup_{j =1,2,\ldots} B^{(j,1)},\qquad{\rm for\ some}\qquad \oo^{(j,1)}\in\tau_b,\
j=1,2,\ldots,
\end{equation}
where  $B^{(j,1)}=\oo^{(j,1)} \setminus \left (\bigcup_{i=1}^{j-1} \oo^{(i,1)} \right)$  are disjoint elements of $\mathcal{A}_2$ (it is allowed that
$\oo^{(j,1)}=\emptyset $ or $B^{(j,1)}=\emptyset$ for some
$j=1,2,\ldots$). 

To prove \eqref{eq:q1EF} for all open sets $\mathcal{O}$ in $\X$, we first show that there exists a subsequence $\{(z^{(n_k)}, a^{(n_k)})\}_{k
=1,2,\ldots}$ of the sequence $\{(z^{(n)}, a^{(n)})\}_{n =1,2,\ldots}$ and a set $C^* \in \B(\Y)$ such that \eqref{eq:q1EF} holds for all $\mathcal{O}\in
\mathcal{A}_2.$ 

Consider an arbitrary $\oo^*\in {\mathcal A}_1.$ Then $\mathcal{\oo}^*=\cup_{i=1}^n \mathcal{\oo}^{(j_i)}$ for some $n=1,2,\ldots,$ where $\oo^{(j_i)}\in
\tau_b,$ $i=1,\ldots,n.$  Let $\mathcal{A}^{(n)}=\left\{\cap_{m=1}^k \oo^{(i_m)} :\,\{i_1,i_2,\ldots,i_k\}\subseteq \{j_1,j_2,\ldots,j_n\} \right\}$ be the
finite set of possible intersections of $\oo^{(j_1)},\ldots,\oo^{(j_n)}.$ The principle of inclusion-exclusion implies that for
$\mathcal{\oo}^*=\cup_{i=1}^n \mathcal{\oo}^{(j_i)},$ $C\in \B(\Y)$, $z,z'\in\P(\X),$ and $a,a'\in \A$,
\begin{equation}\label{e:17}
|R(\oo^* \times C|z, a) - R(\oo^*\times C|z',a')|\le\sum_{B\in\mathcal{A}^{(n)}} |R(B \times C|z, a) - R(B\times C|z',a')|.\end{equation} In view of the
assumption of the lemma regarding finite intersections of the elements of the base $\tau_b$, for each $\oo^*\in {\mathcal A}_1$ the family
$\mathcal{R}_{\oo^*}$ is equicontinuous at all the points $(z,a)\in \P(\X)\times \A$. Inequality (\ref{eq:EC}) implies that for each $B\in \mathcal{A}_2$
the family $\mathcal{R}_{B}$ is equicontinuous at all the points $(z,a)\in \P(\X)\times \A$,
 that is, \eqref{eq:EC1} holds with arbitrary
$\oo_1\in\tau_b$ and $ \oo_2 \in\mathcal{A}_1$. This fact along with the definition of $H$ (see \eqref{3.4}) means that
\begin{equation}
\label{eq:ECBj} \lim_{n \to \infty}\sup_{C\in\B(\Y)}\left|\int_{C} H(B|z^{(n)}, a^{(n)}, y) R'(dy|z^{(n)}, a^{(n)})- \int_{C}H(B|z, a, y) R'(dy|z,a)\right|
= 0,
\end{equation}
for any $B\in\mathcal{A}_2$.


Since the set ${\mathcal A}_2$ is countable, let $\mathcal{A}_2 :=
\{B^{(j)}:\,j=1,2,\ldots\}.$  Denote 
$z^{(n,0)}=z^{(n)}$, $a^{(n,0)}=a^{(n)}$ for all $n=1,2,\ldots.$  For $j=1,2,\ldots$,   from
\eqref{eq:ECBj} and Theorem~\ref{setwise} with $\S=
\Y$, $s = y$, $h^{(n)}(s) = H(B^{(j)}|z^{(n,j-1)}, a^{(n,j-1)}, y)$, $\mu^{(n)}(\cdot) = R'(\,\cdot\,|z^{(n,j-1)},
a^{(n,j-1)})$, $h(s) = H(B^{(j)} |z, a, y)$, and $\mu(\,\cdot\,) 
= R'(\,\cdot\,| z,a)$, 
there exists a subsequence $\{(z^{(n,j)}, a^{(n,j)})\}_{n =1,2,\ldots}$ of the sequence $\{(z^{(n,j-1)}, a^{(n,j-1)})\}_{n =1,2,\ldots}$ and a set
$C^{(*,j)} \in \B(
\Y)$ such that
\begin{equation}
\label{B-j} R'(C^{(*,j)} \mid z,a)  = 1 \mbox{ and } \lim_{n \to \infty}H( B^{(j)} \mid z^{(n,j)}, a^{(n,j)}, y ) = H(B^{(j)} \mid z,a,y) \mbox{ for all }
y \in C^{(*,j)}.
\end{equation}

Let $C^* := \bigcap_{j = 1}^\infty C^{(*,j)}$.  Observe that $R'(C^*| z,a) = 1$. Let $z^{(n_k)}=z^{(k,k)}$ and $a^{(n_k)}=a^{(k,k)},$ $k=1,2,\ldots\ .$ As
follows from Cantor's diagonal argument, \eqref{eq:q1EF} holds with $\oo=B^{(j)}$ for
all $j =1,2,\ldots\ .$ 
In other words, \eqref{eq:q1EF} holds for all $\oo\in {\cal A}_2.$

Let $\oo$ be an arbitrary open set in $\X$ and
$B^{(1,1)},B^{(2,1)},\ldots$ be disjoint elements of ${\cal A}_2$
satisfying \eqref{eq:equi}. Countable additivity of probability
measures $H(\cdot|\cdot,\cdot)$ implies that for all $y \in C^*$ 
\begin{multline*}
\begin{aligned}
\ilim_{k\to \infty}H(\mathcal{O}|z^{(n_k)}, a^{(n_k)},y)
&=\ilim_{k\to \infty}\sum_{j=1}^\infty H(B^{(j,1)}|z^{(n_k)},
a^{(n_k)},y)
\\ \ge
\sum_{j=1}^\infty \ilim_{k\to\infty}H(B^{(j,1)}|z^{(n_k)},
a^{(n_k)},y)& =\sum_{j=1}^\infty
H(B^{(j,1)}|z,a,y)=H(\mathcal{O}|z,a,y),
\end{aligned}
\end{multline*}
where the inequality follows from Fatou's lemma. 
Since $R'(C^* | z,a) =
1$, \eqref{eq:q1EF} holds.
\end{proof}

\begin{proof}[Proof of Theorem~\ref{t:totalvar}]  The setwise continuity
of the stochastic kernel $R'$ follows from  Corollary~\ref{corR'} that states the continuity of $R'$ in the total variation.  The validity of Assumption
({\bf H}) follows from Lemma~\ref{corRsetmin} and Lemma~\ref{R-H}.  In particular, $\tau_b$ is any countable base of the state space $\X$, and, in view of
Lemma~\ref{corRsetmin}, the family ${\mathcal R}_\oo$ is equicontinuous at all the points $(z,a)\in\P(\X)\times\A$  for each open set $\oo$ in $\X,$ which
implies that the assumptions of Lemma~\ref{R-H} hold.
\end{proof}

\section{Preservation of Properties of One-Step Costs
and Proof of Theorem~\ref{th:wstar}}\label{S6}

As shown in this section, the reduction of a POMDP to the COMDP
preserves properties of one-step cost functions that are needed
for the existence of optimal policies.  These properties include
inf-compactness and $\K$-inf-compactness.   In particular, in this
section we prove Theorem~\ref{th:wstar} and thus complete the
proof of Theorem~\ref{main1}.


We recall that an $\overline{\mathbb{R}}$-valued function $f,$
defined on a nonempty subset $U$ of a metric space $\mathbb{U},$
is called \textit{inf-compact on $U$} if all the level sets
$\{y\in U \, : \, f(y)\le \lambda\}$, $\lambda\in\mathbb{R}$, are
compact. A function $f$ is called \emph{lower semi-continuous}, if
all the level sets are closed.

The notion of a $\K$-inf-compact function $c(x,a),$ defined in
Section~\ref{S2} for a function $c:\X\times\A\to
\overline{\mathbb{R}},$ is also applicable to a function
$f:\S_1\times \S_2\to \overline{\mathbb{R}}$, where $\S_1$ and
$\S_2$ are metric spaces, or certain more general toplogical
spaces; see Feinberg et al. \cite{JMAA, JMAA1} for details, where
the properties of $\K$-inf-compact functions are described.
%
%
%
%
In particular, according to Feinberg et al. \cite[Lemma 2.1]{JMAA},
if a function $f$ is inf-compact on $\S_1\times\S_2$ then it is
$\K$-inf-compact on $\S_1\times\S_2$. According to Feinberg et al.
\cite[Lemmas 2.2, 2.3]{JMAA}, a $\K$-inf-compact function $f$ on
$\S_1\times\S_2$ is lower semi-continuous on $\S_1\times\S_2$, and,
in addition, for each $s_1\in\S_1$, the function  $f(s_1,\cdot)$ is
inf-compcat on $\S_2$.

\begin{lemma}\label{lem:lsc}
If the function $c:\X\times\A\to \R$ is bounded below and lower
semi-continuous on $\X\times\A$, then the function
$\c:\P(\X)\times\A\to\R$ defined in (\ref{eq:c}) is bounded below
and lower semi-continuous  on $\P(\X)\times\A$.
\end{lemma}
\begin{proof}
The statement of this lemma directly follows from generalized
Fatou's Lemma~\ref{lem:F}(ii). 
\end{proof}

The inf-compactness of $c$ on $\X\times\A$ implies the
inf-compactness of $\c$ on $\P(\X)\times \A$. We recall that an
inf-compact function on $\X\times\A$  with values in
$\R=\mathbb{R}\cup\{+\infty\}$ is bounded below on $\X\times\A$.

\begin{theorem}
If $c:\X\times\A\to \R$ is an inf-compact function on
$\X\times\A$, then the function $\c:\P(\X)\times\A\to\R$ defined
in (\ref{eq:c}) is inf-compact  on $\P(\X)\times\A$.
\end{theorem}
\begin{proof}
Let $c:\X\times\A\to \R$ be an inf-compact function on $\X\times\A$.
Fix an arbitrary $\lambda\in \mathbb{R}$. To prove that the level
set $\D_{\c}(\lambda;\P(\X)\times\A)=\{(z,a)\in\P(\X)\times\A:\,
{\bar c}(z,a)\le\lambda\}$ is compact, consider an arbitrary
sequence $\{z^{(n)},a^{(n)}\}_{n=1,2,\ldots}\subset
\D_{\c}(\lambda;\P(\X)\times\A)$. It is enough to show that
$\{z^{(n)},a^{(n)}\}_{n=1,2,\ldots}$ has a limit point $(z,a)\in
\D_{\c}(\lambda;\P(\X)\times\A)$.

Let us show that the sequence of probability measures
$\{z^{(n)}\}_{n=1,2,\ldots}$ has a limit point $z\in\P(\X)$. Define 
$\X_{<{+\infty}}:=\X\setminus \X_{+\infty}$, where $
\X_{+\infty}:=\{x\in\X\,:\, c(x,a)={+\infty}\mbox{ for all
}a\in\A\}$.

The inequalities
\begin{equation}\label{eq:inf-comp}
\int_{\X}c(x,a^{(n)})z^{(n)}(dx)\le \lambda,\quad n=1,2,\dots,
\end{equation}
imply that $z^{(n)}(\X_{+\infty})=0$ for any $n=1,2,\dots\ .$  Thus (\ref{eq:inf-comp}) transforms into
\begin{equation}\label{eq:inf-comp1}
\int_{\X_{<{+\infty}}}c(x,a^{(n)})z^{(n)}(dx)\le \lambda, \qquad
n=1,2,\dots\ .
\end{equation}
By definition of inf-compactness, the function $c:\X\times\A\to
\R$ is inf-compact on $\X_{<{+\infty}}\times\A$. According to
Feinberg et al.~\cite[Corollary~3.2]{FKZ},  the real-valued
function 
$\psi(x)=\inf\limits_{a\in\A}c(x,a)$,
$x\in\X_{<{+\infty}}$, with values in $\mathbb{R}$,
is inf-compact on $\X_{<{+\infty}}$. Furthermore,
(\ref{eq:inf-comp1}) implies that $
\int_{\X_{<{+\infty}}}\psi(x)z^{(n)}(dx)\le \lambda,$ 
$n=1,2,\dots\ . $ Thus, Hern\'{a}ndez-Lerma and
Lasserre~\cite[Proposition~E.8]{HLerma1} and Prohorov's
Theorem~\cite[Theorem~E.7]{HLerma1}, yield relative compactness of
the sequence $\{z^{(n)}\}_{n=1,2,\ldots}$ in
$\P(\X_{<{+\infty}})$. Thus there exists a subsequence
$\{z^{(n_k)}\}_{k=1,2,\ldots}\subseteq \{z^{(n)}\}_{n=1,2,\ldots}$
and a probability measure $z\in \P(\X_{<{+\infty}})$ such that
$z^{(n_k)}$ converges to $z$ in $\P(\X_{<{+\infty}})$. Let us set
$z( \X_{+\infty})=0$. As $z^{(n)}(\X_{+\infty})=0$ for all
$n=1,2,\dots$, then the sequence of probability measures
$\{z^{(n_k)}\}_{k=1,2,\ldots}$ converges weakly  and its limit
point $z$ belongs to $\P(\X)$.

The sequence $\{a^{(n_k)}\}_{k=1,2,\ldots}$ has a limit point
$a\in\A$. Indeed,  inequality (\ref{eq:inf-comp}) implies that for
any $k=1,2,\dots$ there exists at least one $x^{(k)}\in\X$ such
that $c(x^{(k)},a^{(n_k)})\le \lambda$. The inf-compactness of
$c:\X\times\A\to \R$ on $\X\times\A$ implies that
$\{a^{(k)}\}_{k=1,2,\ldots}$ has a limit point $a\in\A$. To finish
the proof note that Lemma~\ref{lem:lsc}, generalized
Fatou's Lemma~\ref{lem:F}(ii), 
and (\ref{eq:inf-comp}) imply that $ \int_\X c(x,a)z(dx)\le \lambda. $
\end{proof}



\begin{proof}[Proof of Theorem~\ref{th:wstar}] If $c$ is bounded below on $\X\times\A$, then formula (\ref{eq:c}) implies
that $\c$ is bounded below on $\P(\X)\times\A$ by the same lower
bound as $c.$ Thus, it is enough to prove the $\K$-inf-compactness
of $\c$ on $\P(\X)\times\A$.

Let a sequence of probability measures $\{z^{(n)} \}_{
n=1,2,\ldots}$ on $\X$ weakly converges to $z\in\P(\X)$. Consider
an arbitrary sequence $\{a^{(n)} \}_{ n=1,2,\ldots}\subset \A$
satisfying the condition that the sequence
$\{\c(z^{(n)},a^{(n)})\}_{n=1,2,\ldots}$ is bounded above. Observe
that $\{a^{(n)} \}_{ n=1,2,\ldots}$ has a limit point $a\in \A.$
Indeed, boundedness  below of the $\R$-valued function $c$ on
$\X\times\A$ and generalized Fatou's
Lemma~\ref{lem:F}(ii) 
imply that for some $\lambda<{+\infty}$
\begin{equation}\label{eq:k1}
\int_{\X}\underline{c}(x)z(dx)\le \ilim\limits_{n\to{\infty}}
\int_{\X}c(x,a^{(n)})z^{(n)}(dx)\le\lambda,
\end{equation}
where
\begin{equation}\label{eq:k2}
\underline{c}(x):=\ilim\limits_{y\to x,\,n\to{\infty}}
c(y,a^{(n)}).
\end{equation}
Inequality (\ref{eq:k1}) implies the existence of $x^{(0)}\in\X$
such that $\underline{c}(x^{(0)})\le\lambda$. Therefore, formula
(\ref{eq:k2}) implies the existence of a subsequence $\{a^{(n_k)}
\}_{k=1,2,\ldots}\subseteq\{a^{(n)} \}_{ n=1,2,\ldots}$ and a
sequence $\{y^{(k)} \}_{k=1,2,\ldots}\subset\X$ such that $
y^{(k)}\to x^{(0)}$ as $k\to{\infty}$ and $
c(y^{(k)},a^{(n_k)})\le \lambda+1$
for $k=1,2,\ldots\ .$  
Since $c:\X\times\A\to \R$ is $\K$-inf-compact on $\X\times\A$,
the sequence $\{a^{(n_k)} \}_{k=1,2,\ldots}$ has a limit point
$a\in\A$, which is the limit point of the initial sequence
$\{a^{(n)} \}_{ n=1,2,\ldots}$. Thus, the function
$\bar{c}$ is $\K$-inf-compact on $\P(\X)\times\A.$ 
\end{proof}

 Arguments similar to the proof of Theorem~\ref{th:wstar} imply
 the
inf-compactness of $\c(z,a)$ in $a\in\A$ for any $z\in\P(\X)$, if
$c(x,a)$ is inf-compact in $a\in\A$ for any $x\in\X$.

\begin{theorem}\label{cor:inf-HL}
If the function $c(x,a)$ is inf-compact in $a\in\A$ for each
$x\in\X$ and bounded below on $\X\times\A$, then the function
$\c(z,a)$ 
is inf-compact in $a\in\A$ for each $z\in\P(\X)$ and bounded 
below on $\P(\X)\times\A$.
\end{theorem}
\begin{proof}  Fix $z\in\P(\X)$ and consider a sequence $\{a^{(n)}\}_{n=1,2,\ldots}$
in $\A$ such that $c(z,a^{(n)})\le \lambda$ for some
$\lambda<\infty,$  ${n=1,2,\ldots}\ .$  The classic Fatou's lemma
implies that (\ref{eq:k1}) holds with $z^{(n)}=z$, ${n=1,2,\ldots},$
and $\underline{c}(x)=\liminf_{n\to{\infty}} c(x,a^{(n)}),$
$x\in\X.$ Thus, there exists $x^{(0)}\in\X$ such that
$\liminf_{n\to{\infty}} c(x^{(0)},a^{(n)})\le \lambda$. This
together with the inf-compactness of $c(x^{(0)},a)$ in $a\in\A$
implies that the sequence $\{a^{(n)}\}_{n=1,2,\ldots}$ has a limit
point in $\A$.
\end{proof}

\begin{proof}[Proof of Theorem~\ref{main1}]
Theorem~\ref{main1} follows from Theorems~\ref{teor4.3}, \ref{th:wstar}, and \ref{th:contqqq2}.
\end{proof}


\section{Combining Assumption~{\bf {\bf(H)}} and the Weak
Continuity~of~$H$}\label{S7}

Theorem~\ref{main1} assumes either the weak continuity of $H$ or Assumption~{\bf {\bf(H)}} together with the setwise continuity of $R'$. For some
applications, see e.g., Subsection~\ref{S8.2} that deals with inventory control, the filtering kernel $H$ satisfies Assumption~{\bf {\bf(H)}} for some
observations and it is weakly continuous for other observations. The following theorem is applicable to such situations.

\begin{theorem}\label{cor:main}
Let the observation space $\Y$ be partitioned into two disjoint
subsets $\Y_1$ and $\Y_2$ such that $\Y_1$ is open in $\Y$. If the
following assumptions hold:
\begin{enumerate}[(a)]
\item the stochastic kernels $P$ on $\X$ given $\X\times\A$ and
$Q$ on $\Y$ given $\A\times\X$ are weakly continuous; \item the
measure $R'(\,\cdot\,|z,a)$ on $(\Y_2,\B(\Y_2))$ is setwise
continuous in $(z,a)\in \P(\X)\times \A,$ that is, for every
sequence $\{(z^{(n)},a^{(n)})\}_{n=1,2,\ldots}$  in
$\P(\X)\times\A$ converging to $(z,a)\in\P(\X)\times\A$  and for
every $C\in\B(\Y_2),$ we have
$R'(C|z^{(n)},a^{(n)})\to R'(C|z,a);$ 

 \item there exists a stochastic kernel
$H$ on $\X$ given $\P(\X)\times\A\times\Y$ satisfying (\ref{3.4})
such that:

(i) the stochastic kernel $H$ on $\X$ given
$\P(\X)\times\A\times\Y_1$ is weakly continuous;

(ii)  Assumption {\rm\bf({\bf H})} holds on $\Y_2$, that is, if a
sequence $\{z^{(n)}\}_{n=1,2,\ldots}\subseteq\P(\X)$ converges
weakly to $z\in\P(\X)$ and a sequence
$\{a^{(n)}\}_{n=1,2,\ldots}\subseteq\A$ converges to $a\in\A$,
then there exists a subsequence $\{(z^{(n_k)},a^{(n_k)})\}_{k=1,2,\ldots}\subseteq
\{(z^{(n)},a^{(n)})\}_{n=1,2,\ldots}$ and a measurable subset $C$ of $\Y_2$ such that
$R'(\Y_2\setminus C|z,a)=0$ and  
$H(z^{(n_k)},a^{(n_k)},y)$  converges weakly to
$H(z,a,y)$ for all $y\in C$; 
\end{enumerate}
then the stochastic kernel $q$ on $\P(\X)$ given $\P(\X)\times\A$
is weakly continuous.  If, in addition to the above conditions,
assumptions (a) and (b) from Theorem~\ref{main1} hold, then the
COMDP $(\P(\X),\A,q,\c)$ satisfies Assumption {\rm\bf(${\rm \bf
W^*}$)} and therefore statements (i)--(vi) of
Theorem~\ref{teor4.3} hold.

\end{theorem}



\begin{proof} 
%
The stochastic kernel $q(dz'|z,a)$ on $\P(\X)$ given $\P(\X)\times \A$ is weakly continuous if and only if for every open set $D$ in $\P(\X)$ the function
$q(D|z,a)$ is lower semi-continuous in $(z,a)\in \P(\X)\times \A$; Billingsley \cite[Theorem 2.1]{Bil}.  Thus, if $q$ is not weakly continuous,
 there exist an open set $D$ in $\P(\X)$ and  sequences
$z^{(n)}\to z$ weakly  and  $a^{(n)}\to a$, where $z,z^{(n)}\in \P(\X)$ and $a,a^{(n)}\in\A$, $n=1,2,\ldots,$  such that
\begin{equation*}
\ilim_{n\to\infty}q(D|z^{(n)},a^{(n)})< q(D|z,a).
\end{equation*}
%
Then there exists $\varepsilon^*>0$ and a subsequence $\{z^{(n,1)},a^{(n,1)}\}_{n=1,2,\ldots}\subseteq\{z^{(n)},a^{(n)}\}_{n=1,2,\ldots}$  such that for
all $n=1,2,\ldots$
\begin{multline}\label{ceq:q2111}
\int_{\Y_1}\h\{H(z^{(n,1)},a^{(n,1)},y)\in D\}R'(dy|z^{(n,1)},a^{(n,1)})+\int_{\Y_2}\h\{H(z^{(n,1)},a^{(n,1)},y)\in D\}R'(dy|z^{(n,1)},a^{(n,1)})\\
\qquad=q(D|z^{(n,1)},a^{(n,1)})\le q(D|z,a)-\varepsilon^*\quad\quad\quad\quad\quad\quad\quad\quad\quad\quad\quad\quad\quad\quad\quad\quad\quad\quad
\end{multline}\noindent \[ =\int_{\Y_1}\h\{H(z,a,y)\in D\}R'(dy|z,a)+\int_{\Y_2}\h\{H(z,a,y)\in
D\}R'(dy|z,a)-\varepsilon^*,
\]
where the stochastic kernel $H$ on $\X$ given
$\P(\X)\times\A\times\Y$ satisfies (\ref{3.4}) and assumption (c)
of Theorem~\ref{cor:main}.

Since $\Y_1$ is an open set in $\Y$ 
%
and the stochastic kernel $H$ on $\X$ given $\P(\X)\times\A\times\Y_1$ is weakly continuous, for all $ y\in \Y_1 $
\begin{equation}\label{eq:corm3}
\ilim_{\substack{n\to\infty \\ y'\to y}}\h\{H(z^{(n,1)},a^{(n,1)},y')\in D\}= \ilim_{\substack{n\to\infty \\ y'\to y,\,y'\in \Y_1}}
\h\{H(z^{(n,1)},a^{(n,1)},y')\in D\}\ge \h\{H(z,a,y)\in D\}.
\end{equation}

The weak continuity of the stochastic kernels $P$ and $Q$ on $\X$
given $\X\times\A$ and on $\Y$ given $\A\times\X$ respectively
imply the weak continuity of the stochastic kernel $R'$ on $\Y$ given $\P(\X)\times\A$; Hern\'{a}ndez-Lerma \cite[p.~92]{HLerma}. Therefore,
\begin{equation}
\begin{aligned}\label{eq:corm1}
&\ilim\limits_{n\to\infty}\int_{\Y_1}\h\{H(z^{(n,1)},a^{(n,1)},y)\in D\}R'(dy|z^{(n,1)},a^{(n,1)})\\
&\ge \int_{\Y_1}\ilim\limits_{n\to\infty, \, y'\to y}\h\{H(z^{(n,1)},a^{(n,1)},y')\in D\}R'(dy|z^{(n,1)},a^{(n,1)})\\ &\ge \int_{\Y_1}\h\{H(z,a,y)\in
D\}R'(dy|z,a),
\end{aligned}
\end{equation}
where the first inequality follows from Lemma~\ref{lem:F}(ii) and the second one follows from formula (\ref{eq:corm3}).

The inequality
\begin{equation}\label{eq:corm2}
\slim\limits_{n\to\infty}\int_{\Y_2}\h\{H(z^{(n,1)},a^{(n,1)},y)\in
D\}R'(dy|z^{(n,1)},a^{(n,1)})\ge \int_{\Y_2}\h\{H(z,a,y)\in
D\}R'(dy|z,a)
\end{equation}
together with (\ref{eq:corm1}) 
contradicts (\ref{ceq:q2111}).  This contradiction implies that 
$q(\,\cdot\,|z,a)$ is a weakly continuous stochastic kernel on
$\P(\X)$ given $\P(\X)\times\A$.

To complete the proof of Theorem~\ref{cor:main}, we prove inequality (\ref{eq:corm2}). If $R'(\Y_2|z,a)=0$, then inequality (\ref{eq:corm2}) holds. Now let
$R'(\Y_2|z,a)>0$. Since $R'(\Y_2|z^{(n,1)},a^{(n,1)})\to R'(\Y_2|z,a)$ as $n\to\infty$, there exists $N=1,2,\ldots$ such that
$R'(\Y_2|z^{(n,1)},a^{(n,1)})>0$ for any $n\ge N$. We endow $\Y_2$ with the same metric as in $\Y$ and set
\[
R_1'(C|z',a'):=\frac{R'(C|z',a')}{R'(\Y_2|z',a')}, \quad z'=z,z^{(n,1)},\ a'=a,a^{(n,1)}, \ n\ge N,\ C\in \B(\Y_2).
\]
Assumption (b) of Theorem~\ref{cor:main} means that the stochastic
kernel $R_1'(dy|z,a)$ on $\Y_2$ given $\P(\X)\times\A$ is setwise
continuous. Assumption (ii) of Theorem~\ref{cor:main} implies the 
existence of a subsequence
$\{z^{(n,2)},a^{(n,2)}\}_{n=1,2,\ldots}\subseteq
\{z^{(n,1)},a^{(n,1)} \}_{n=1,2,\ldots}$ and a measurable subset
$C$ of $\Y_2$ such that $R_1'(\Y_2\setminus C|z,a)=0$ and $
H(z^{(n,2)},a^{(n,2)},y)$ converges weakly to $H(z,a,y)$ as
$n\to\infty$ for all $y\in C$.
Therefore, since $D$ is an open set in $\P(\X)$, we have
\begin{equation}\label{ceq:q0000}
\ilim_{k\to\infty}\h\{H(z^{(n,2)},a^{(n,2)},y)\in D\}\ge \h\{H(z,a,y)\in D\},\qquad y\in C.
\end{equation}
Formulas (\ref{3.7}), (\ref{ceq:q0000}), the setwise continuity of the stochastic kernel $R_1'$ on $\Y_2$ given $\P(\X)\times\A$, and Lemma~\ref{lem:F}(i)
imply 
\begin{multline*}
\frac{1}{R'(\Y_2|z,a)} \ilim\limits_{k\to\infty}\int_{\Y_2}\h\{H(z^{(n,2)},a^{(n,2)},y)\in D\}R'(dy|z^{(n,2)},a^{(n,2)})\\ \ge
\ilim\limits_{k\to\infty}\frac{\int_{\Y_2}\h\{H(z^{(n,2)},a^{(n,2)},y)\in D\}R'(dy|z^{(n,2)},a^{(n,2)})}{R'(\Y_2|z^{(n,2)},a^{(n,2)})}\ge
\frac{\int_{\Y_2}\h\{H(z,a,y)\in D\}R'(dy|z,a)}{R'(\Y_2|z,a)},
\end{multline*}
and thus (\ref{eq:corm2}) holds.
\end{proof}


\begin{corollary}\label{cor:main1}
Let the observation space $\Y$ be partitioned into two disjoint subsets $\Y_1$ and $\Y_2$ such that $\Y_1$ is  open in $\Y$ and $\Y_2$ is countable. If the
following assumptions hold:
\begin{enumerate}[(a)]
\item the stochastic kernels $P$ on $\X$ given $\X\times\A$ and
$Q$ on $\Y$ given $\A\times\X$ are weakly continuous; \item
$Q(y|a,x)$ is a continuous function on $\A\times\X$ for each $y\in
\Y_2;$ \item there exists a stochastic kernel $H$ on $\X$ given
$\P(\X)\times\A\times\Y$ satisfying (\ref{3.4}) such that the
stochastic kernel $H$ on $\X$ given $\P(\X)\times\A\times\Y_1$ is
weakly continuous;
\end{enumerate}
then assumptions (b) and (ii) of Theorem~\ref{cor:main} hold, and the stochastic kernel $q$ on $\P(\X)$ given $\P(\X)\times\A$
is weakly continuous.  If, in addition to the above conditions,
assumptions (a) and (b) from Theorem~\ref{main1} hold, then the
COMDP $(\P(\X),\A,q,\c)$ satisfies Assumption {\rm\bf(${\rm \bf
W^*}$)} and therefore statements (i)--(vi) of
Theorem~\ref{teor4.3} hold.
\end{corollary}




\begin{proof} To prove the corollary,
it is sufficient to verify  conditions (b) and (ii) of
Theorem~\ref{cor:main}. 
For each $B\in\B(\X)$ and for each $y\in\Y_2$, Hern\'{a}ndez-Lerma
\cite[Proposition~C.2(b), Appendix~C]{HLerma}, being  repeatedly
applied to formula \eqref{3.3} with $C=\{y\},$ implies  the
continuity of $R(B\times\{y\}|z,a)$ in $(z,a)\in\P(\X)\times\A$.
 In particular, the function $R'(y|\,\cdot\, ,\,
\cdot\,)$ is continuous on $\P(\X)\times\A$.  If $R'(\,y\,|z,a)>0$
then, in view of \eqref{3.4},
$H(B|z,a,y)=R(B\times\{y\}|z,a)/R'(\,y\,|z,a)$, and, if $y$ is
fixed, this function is continuous at the point $(z,a).$ Thus,
condition (ii) of Theorem~\ref{cor:main} holds.  Since the set
$\Y_2$ is closed in $\Y$, the function $Q(\Y_2|a,x)$ is upper
semi-continuous in $(a,x)\in \A\times\X.$
  Generalized Fatou's Lemma~\ref{lem:F}, being
repeatedly applied to \eqref{3.5} with $C=\Y_2$, implies that
$R'(\Y_2|z,a)$ is upper semi-continuous in
$(z,a)\in\P(\X)\times\A$.  This implies that, for every
$Y\subseteq\Y_2$ and for every sequence $\{
z^{(n)},a^{(n)}\}_{n=1,2,\ldots}\subset \P(\X)\times\A$ converging
to $(z,a)\in \P(\X)\times\A,$
\[|R'(Y|z^{(n)},a^{(n)})-R'(Y|z,a)|\le
\sum_{y\in\Y_2}|R'(y|z^{(n)},a^{(n)})-R'(y|z,a)|\to 0\quad {\rm
as} \quad n\to\infty,
\]
where the convergence takes place because of the same arguments as
in the proof of Corollary 3.8.  Thus, condition (b) of
Theorem~\ref{cor:main} holds.
\end{proof}


\section{Examples of Applications}\label{S8}
To illustrate theoretical results, they are applied in this
section to three particular models: (i) problems defined by
stochastic equations; see Striebel~\cite{Strieb},
Bensoussan~\cite{Bens2}, and Hern\'{a}ndez-Lerma
\cite[p.~83]{HLerma}, (ii) inventory control, and (iii) Markov
Decision Model with incomplete information.

\subsection{Problems Defined by Stochastic Equations} \label{S8.1}
Let $\{\xi_t\}_{t=0,1,\ldots}$ be a sequence of identically
distributed finite random variables with values in $\mathbb{R}$
and with the distribution $\mu.$
 Let $\{\eta_t\}_{t=0,1,\ldots}$ be a sequence of
random variables uniformly distributed on $(0,1)$. An initial
state $x_0$ is a  random variable with values in $\mathbb{R}$. It
is assumed that the random variables $x_0,\xi_0,
\eta_0,\xi_1,\eta_1,\ldots$ are defined on the same probability space and mutually independent. 

Consider a stochastic partially observable control system
\begin{equation}\label{eqe1}
x_{t+1}=F(x_{t},a_{t},\xi_{t}),\,\, t=0,1,\dots,
\end{equation}
\begin{equation}\label{eqe2}
y_{t+1}=G(a_{t},x_{t+1},\eta_{t+1}),\,\, t=0,1,\dots,
\end{equation}
where $F$ and $G$ are given measurable functions from
$\mathbb{R}\times\mathbb{R}\times \mathbb{R}$ to $\mathbb{R}$ and
from $\mathbb{R}\times\mathbb{R}\times (0,1)$ to $\mathbb{R}$
respectively. The initial observation is $y_0=G_0(x_0,\eta_0),$
where $G_0$ is a measurable function from $\mathbb{R}\times (0,1)$
to $\mathbb{R}$. The states $x_t$ are not observable, while the
states $y_t$ are observable. The goal is to minimize the expected
total discounted costs.

Instead of presenting formal definitions of functions $a_t$, we
describe the above problem  as a POMDP with the state space
$\X=\mathbb{R}$, observation space $\Y=\mathbb{R}$, and action
space $\A=\mathbb{R}$. The transition law is
\begin{equation}\label{ex1dp}
P(B|x,a)=\int_{\mathbb{R}} \h\{F(x,a,s)\in B\}\mu(ds),\quad B\in
\mathcal{B}(\mathbb{R}),\ x\in\mathbb{R}, \ a\in\mathbb{R}.
\end{equation}
The observation kernel is
\[
Q(C|a,x)=\int_{(0,1)} \h\{G(a,x,s)\in C\}\lambda(ds),\quad C\in
\mathcal{B}(\mathbb{R}),\ a\in\mathbb{R}, \ x\in\mathbb{R},
\]
where $\lambda\in \P((0,1))$ is the Lebesgue measure on $(0,1)$.
The initial state distribution $p$ is the distribution of the
random variable $x_0$, and the initial observation kenel
$Q_0(C|x)=\int_{(0,1)}{\bf I}\{G_0(x,s)\in C\}\lambda(ds)$ for all
$C\in\B(\R)$ and for all $x\in\X.$

Assume that  $(x,a)\to F(x,a,s)$ is a continuous mapping on
$\mathbb{R}\times\mathbb{R}$ for $\mu$-a.e. $s\in \mathbb{R}$.
Then the stochastic kernel $P(dx'|x,a)$ on $\mathbb{R}$ given
$\mathbb{R}\times\mathbb{R}$ is weakly continuous;
Hern\'{a}ndez-Lerma \cite[p.~92]{HLerma}. 

Assume that: (i) $G$ is a continuous mapping on
$\mathbb{R}\times\mathbb{R}\times (0,1)$, (ii)  the partial
derivative  $g(x,y,s)=\frac{\partial G(x,y,s)}{\partial s}$ exists
everywhere and is  continuous, and (iii) there exists a constant
$\beta>0$ such that $|g(a,x,s)|\ge \beta$ for all
$a\in\mathbb{R}$, $x\in\mathbb{R}$, and $s\in (0,1)$.
Denote by $\g$ the inverse function for $G$ with respect the last variable. Assume that $\g$ is continuous.


Let us prove that under these assumptions the 
observation kernel $Q$ on $\mathbb{R}$ given $\mathbb{R}\times\mathbb{R}$ is continuous in the total variation.
For each $\varepsilon \in (0,\frac12)$, for each Borel set $C\in
\B(\mathbb{R})$, and for all
$(a',x'),(a,x)\in\mathbb{R}\times\mathbb{R}$
\begin{equation*}
\begin{aligned}
&\left|Q(C|a',x')-Q(C|a,x)\right| = \left|\int_{0}^{1}
\h\{G(a',x',s)\in C\}\lambda(ds)-\int_{0}^{1} \h\{G(a,x,s)\in
C\}\lambda(ds)\right|
\\
&\le 4\varepsilon+ \left|\int_{\varepsilon}^{1-\varepsilon}
\h\{G(a',x',s)\in
C\}\lambda(ds)-\int_{\varepsilon}^{1-\varepsilon} \h\{G(a,x,s)\in
C\}\lambda(ds)\right|
\\
&=4\varepsilon+  \left|\int_{G(a',x',[\varepsilon,1-\varepsilon])} \frac{\h\{\tilde{s}\in
C\}}{g(a',x',\g(a',x',\tilde{s}))}\tilde{\lambda}(d\tilde{s})-\int_{G(a,x,[\varepsilon,1-\varepsilon])} \frac{\h\{\tilde{s}\in
C\}}{g(a,x,\g(a,x,\tilde{s}))}\tilde{\lambda}(d\tilde{s})\right|
\\
&\le 4\varepsilon+ \frac{\left|G(a',x',\varepsilon)- G(a,x,\varepsilon)\right|+\left|G(a',x',1-\varepsilon)- G(a,x,1-\varepsilon)\right|}{\beta}
\\
&+\frac{1}{\beta^2}\int_{G(a,x,[\varepsilon,1-\varepsilon])\cap G(a',x',[\varepsilon,1-\varepsilon])}
\left|{g(a',x',\g(a',x',\tilde{s}))}-{g(a,x,\g(a,x,\tilde{s}))}\right|\tilde{\lambda}(d\tilde{s}),
\end{aligned}
\end{equation*}
where $\tilde{\lambda}$ is the Lebesgue measure on $\mathbb{R}$,
the second equality holds because of the changes
$\tilde{s}=G(a',x',s)$ and $\tilde{s}=G(a,x,s)$ in the
corresponding integrals, and the second inequality follows from
direct estimations. Since, the function $G$ is  continuous, %
$G(a',x',\varepsilon)\to G(a,x,\varepsilon)$ and
$G(a',x',1-\varepsilon)\to G(a,x,1-\varepsilon)$ as
$(a',x')\to(a,x)$, for any $(a,x,\varepsilon)\in
\mathbb{R}\times\mathbb{R}\times (0,\frac12)$.  Thus, if
\begin{equation}\label{eq:DefddD}
\int_{\mathbb{R}} D(a,x,a',x',\varepsilon,\tilde{s})\tilde{\lambda}(d\tilde{s})\to 0\qquad {\rm as}\ (a',x')\to (x,a),
\end{equation}
where
$D(a,x,a',x',\varepsilon,\tilde{s}):=\left|{g(a',x',\g(a',x',\tilde{s}))}-{g(a,x,\g(a,x,\tilde{s}))}\right|$,
when $\tilde{s}\in G(a,x,[\varepsilon,1-\varepsilon])\cap
G(a',x',[\varepsilon,1-\varepsilon])$, and
$D(a,x,a',x',\varepsilon,\tilde{s})=0$ otherwise, then
\[\lim\limits_{(a',x')\to(a,x)}\sup\limits_{C\in\B(\mathbb{R})}\left|Q(C|a',x')-Q(C|a,x)\right|= 0.\]

So, to complete the proof of the continuity in the total variation
of the observation kernel $Q$ on $\mathbb{R}$ given
$\mathbb{R}\times\mathbb{R}$, it is sufficient to verify
\eqref{eq:DefddD}. We fix an arbitrary vector
$(a,x,\varepsilon)\in\mathbb{R}\times\mathbb{R}\times(0,\frac12)$
and consider arbitrary converging sequences
 $a^{(n)}\to a$ and $x^{(n)}\to x$. 
Let
$(a',x')=(a^{(n)},x^{(n)})$, $n=1,2,\ldots\ .$ Since the sets
$K:=\{(a^{(n)},x^{(n)})\,:\, n=1,2,\ldots\}\cup\{(a,x)\}$ and
$[\varepsilon,1-\varepsilon]$ are compact and the function $g$ is
continuous on $\mathbb{R}\times\mathbb{R}\times (0,1)$, the
function $|g|$ is bounded above on the compact set $K\times
[\varepsilon,1-\varepsilon]$ by a positive constant $M$. Thus, the
integrand in \eqref{eq:DefddD} is bounded above by $2M$ on the
compact set $G(K\times [\varepsilon, 1-\varepsilon])$ and is equal
to  $0$ on its complement. Since $G$, $g,$ and $\g$ are continuous
functions, for each ${\tilde s}\in\mathbb{R}$ the integrand in
\eqref{eq:DefddD} converges
 to 0 as  $(a',x')\to (a,x)$. Therefore, \eqref{eq:DefddD}
follows from the dominated convergence theorem, because the
Lebeasgue measure of the set $G(K\times [\varepsilon,
1-\varepsilon])$ is finite since this set is compact.

Finally, we assume that the one-period cost function $c:\mathbb{R}\times\mathbb{R}\to\R$ is bounded below and $\K$-inf-compact. Thus, the assumptions  of
Theorem~\ref{main} are satisfied.
   Therefore, for this COMDP there exists a stationary optimal policy, the optimality
equations hold, and  value  iterations converge to the optimal
value.

We remark that the one-dimensional Kalman filter in discrete time
satisfies the above assumptions. In this case,
$F(x_t,a_t,\xi_t)=d^*x_t+b^*a_t+\xi_t$ and
$G(a_t,x_{t+1},\eta_{t+1})=h^*x_{t+1}+c^*\Phi^{-1}(\eta_{t+1}),$
where $c^*\ne 0$ and $\Phi^{-1}$ is the inverse to the cumulative
distribution function of a standard normal distribution
($\Phi^{-1}(\eta_{t+1})$ is a standard normal random variable). In
particular,
$|g(a,x,s)|=|c^*|(2\pi)^{\frac{1}{2}}e^{\frac{\Phi^{-1}(s)^2}{2}}\ge
|c^*|(2\pi)^{\frac{1}{2}} >0$ for all
$s\in [0,1]$. 
Thus, if the cost function $c(x,a)$ is $\K$-inf-compact, then the
conclusions of Theorem~\ref{main} hold for the Kalman filter.  In
particular, the quadratic cost function $c(x,a)=c_1x^2+c_2a^2$ is
$\K$-inf-compact if $c_1\ge 0$ and $c_2>0$.  Thus, the linear
quadratic Gaussian control problem is a particular case of this
model. The one-step cost functions $c(x,a)=(a-x)^2$ and
$c(x,a)=|x-a|,$ which are typically used for identification
problems, are also $\K$-inf-compact.  However, these two functions
are not inf-compact. This illustrates the usefulness of the notion
of $\K$-inf-compactness.

\subsection{Inventory Control with Incomplete Records}\label{S8.2}  This example is
motivated by Bensoussan et al. \cite{Bens3}--\cite{Bens5}, where
several inventory  control  problems for periodic review systems,
when the Inventory Manager (IM) may not have complete information
about inventory levels, are studied. In Bensoussan et
al.~\cite{Bens3}, \cite{Bens5}, a  problem with backorders is
considered.  In the model considered in \cite{Bens3}, the IM does
not know the inventory level, if it is nonnegative, and the IM
knows the inventory level, if it is negative. In the model
considered in \cite{Bens5},  the IM only knows  whether the
inventory level is negative or nonnegative.  In \cite{Bens4} a
problem with lost sales is studied, when the IM only knows whether
a lost sale took place or not. The underlying mathematical
analysis is summarized in \cite{BensOx}, where additional
references can be found.  The analysis includes transformations of
density functions of demand distributions.

The current example studies periodic review systems with backorders and lost sales, when some inventory levels are observable and some are not.  The goal
is to minimize the expected total costs.  Demand distribution  may not have densities.

In the case of full observations, we model the problem as an MDP
with a state space $\X=\mathbb{R}$ (the current inventory level),
action space $\A=\mathbb{R}$ (the ordered amount of inventory),
and action sets $\A(x)=\A$ available at states $x\in \X$. If in a
state $x$ the amount of inventory $a$ is ordered, then the
holding/backordering cost $h(x)$,  ordering cost $C(a),$ and lost
sale cost $G(x,a)$ are incurred, where it is assumed that $h,$
$C,$ and $G$ are
nonnegative lower semi-continuous functions 
with values in ${\R}$ and $C(a)\to +\infty$ as $|a|\to \infty.$
Observe that the one-step cost function $c(x,a)=h(x)+C(a)+G(x,a)$
is $\K$-inf-compact on $\X \times \A$.  Typically $G(x,a)=0$ for
$x\ge 0.$

Let $D_t, t = 0,1,\ldots,$ be i.i.d. random variables with the distribution function $F_D$, where $D_t$ is the demand at epoch $t=0,1,\ldots\ .$  The
dynamics of the system is defined by $x_{t+1}=F(x_t,a_t,D_t),$ where $x_t$ is the current inventory level  and $a_t$ is the ordered (or scrapped) inventory
at epoch
$t=0,1,\ldots\ .$ 
For problems with backorders $F(x_t,a_t,D_t)=x_t+a_t-D_t$ and for
problems with lost sales $F(x_t,a_t,D_t)=|x_t+a_t-D_t|^+$.  In
both cases, $F$ is a continuous function defined on
$\mathbb{R}^3$.  To simplify and unify the presentation, we do not
follow the common agreement that
 $\X=[0,\infty)$
for models with lost sales. 
However,
 for problems with lost sales it is assumed that the initial state distribution $p$ is concentrated on $[0,\infty)$, and this implies that states $x<0$ will never
be visited. We assume that the distribution function $F_D$ is atomless (an equivalent assumption is that the function $F_D$ is continuous).
The state transition law $P$ on $\X$ given $\X\times\A$ is
\begin{equation}
\label{eq:P} P(B|x,a)=\int_{\mathbb{R}}\h\{F(x,a,s)\in B\}dF_D(s),\quad B\in \B(\X),\ x\in\X, \ a\in\A.
\end{equation}
Since we do not assume that demands are nonnegative, this model also covers cash balancing problems and problems with returns; see Feinberg and
Lewis~\cite{FL} and references therein.  In a particular case, when $C(a)=+\infty$ for $a<0$,  orders with negative sizes are infeasible, and, if an order
is placed, the ordered amount of inventory should be positive.

As mentioned above, some states (inventory levels) $x\in\X=\mathbb{R}$ are observable and some are not.  Let inventory be stored in containers. From a
mathematical prospective, containers are  elements of a finite or countably infinite partition of $\X=\mathbb{R}$ into disjoint convex sets, and each of
these sets is not a singleton. In other words, each
 container $B_{i+1}$  is an interval (possibly open, closed, or
semi-open) with ends $d_i$ and $d_{i+1}$ such that $-\infty\le
d_i<d_{i+1}\le +\infty$, and the union of these disjoint intervals
is $\mathbb{R}.$ In addition, we assume that
$d_{i+1}-d_i\ge\gamma$ for some constant $\gamma>0$ for all
containers, that is, the sizes of all the containers are uniformly
bounded below by a positive number.  We also follow an agreement
that the 0-inventory level belongs to a container with end points
$d_0$ and $d_1$, and a container with end points $d_i$ and
$d_{i+1}$ is labeled as the $(i+1)$-th container $B_{i+1}$. Thus,
container $B_1$ is the interval in the partition containing point
0. Containers' labels can be nonpositive.  If there is a container
with the smallest (or largest) finite label $n$ then
$d_{n-1}=-\infty$ (or $d_n=+\infty$, respectively).  If there are
containers with  labels $i$ and $j$ then there are containers with
all the labels between $i$ and $j$. In addition each container is
either transparent or nontransparent. If the inventory level $x_t$
belongs to a nontransparent container, the IM only  knows which
container the inventory level belongs to. If an inventory level
$x_t$ belongs to a transparent container, the IM knows that the
amount of inventory is exactly $x_t$.

For each nontransparent container with end points $d_i$ and
$d_{i+1}$, we fix an arbitrary point $b_{i+1}$ satisfying
$d_i<b_{i+1}<d_{i+1}$. For example, it is possible to set
$b_{i+1}=0.5d_i+0.5d_{i+1},$ when $\max\{|d_i|, |d_{i+1}|\}<\infty.$
If an inventory level belongs to a nontransparent container $B_{i}$,
the IM observes $y_t=b_i.$ Let $L$ be the set of labels of the
nontransparent containers.  We set $Y_L=\{b_i\,:\,i\in L\}$ and
define the observation set $\Y=\T\cup Y_L$, where $\T$ is the union
of all transparent containers $B_i$ (transparent elements of the
partition).  If the observation $y_t$ belongs to a transparent
container (in this case, $y_t\in\T$), then the IM knows that the
inventory level $x_t=y_t$. If $y_t\in Y_L$ (in this case, $y_t=b_i$
for some $i$), then the IM knows that the inventory
level belongs to the container $B_i$, 
and this container is nontransparent. Of course, the distribution
of this level can be computed. 

Let $\rho$ be the Euclidean distance on $\mathbb{R}:$ $\rho
(a,b)=|a-b|$ for $a,b\in \Y$. On the state space $\X=\mathbb{R}$ we
consider the metric $\rho_\X(a,b)=|a-b|,$ if $a$ and $b$ belong to
the same container, and $\rho_\X(a,b)=|a-b|+1$ otherwise, where
$a,b\in \X$. The space $(\X,\rho_\X)$ is a Borel subset of a Polish
space (consisting of closed containers, that is, each finite point
$d_i$ is represented by two points: one belonging to the container
$B_i$ and another one to the container $B_{i+1}$).
 We notice that
 $\rho_\X(x^{(n)},x)\to 0$ as $n\to\infty$  if and only if $|x^{(n)}-x|\to
0$ as $n\to\infty$ and the sequence $\{x^{(n)}\}_{n=N,N+1,\ldots}$
belongs to the same container as $x$ for a sufficiently large $N$.
Thus, convergence on $\X$ in the metric $\rho_{\X}$ implies
convergence in the Euclidean metric.  In addition, if $x\ne d_i$
for all containers $i$, then $\rho_\X(x^{(n)},x)\to 0$ as
$n\to\infty$ if and only if $|x^{(n)}-x|\to 0$ as $n\to\infty.$
Therefore, for any open set $B$ in $(\X,\rho_\X)$, the set
$B\setminus(\cup_i\{d_i\})$ is open in $(\X,\rho).$  We notice
that each container $B_i$ is an open and closed set in
$(\X,\rho_\X).$

Observe that the state transition law $P$ given by \eqref{eq:P} is weakly continuous in $(x,a)\in\X\times\A$. Indeed, let $B$ be an open set in
$(\X,\rho_\X)$ and  $\rho_\X(x^{(n)},x)\to 0$ and $|a^{(n)}-a|\to 0$ as $n\to\infty$. The set $B^\circ:=B\setminus(\cup_i\{d_i\})=B\cap (\cup_i
(d_i,d_{i+1}))$ is open in $(\X,\rho).$ Since $F$ (as a function from $(\X,\rho_\X)\times(\A,\rho)\times(\mathbb{R},\rho)$ into $(\X,\rho_\X)$) is a
continuous function in the both models, with backorders and lost sales, Fatou's lemma yields
\begin{multline*}
\liminf_{n\to\infty}P(B^\circ|x^{(n)},a^{(n)})=\liminf_{n\to\infty}\int_{\mathbb{R}}\h\{F(x^{(n)},a^{(n)},s)\in B^\circ\}dF_D(s)\\ \ge
\int_{\mathbb{R}}\liminf_{n\to\infty}\h\{F(x^{(n)},a^{(n)},s)\in B^\circ\}dF_D(s)\ge \int_{\mathbb{R}}\h\{F(x,a,s)\in B^\circ\}dF_D(s)= P(B^\circ|x,a).
\end{multline*}
Therefore, $ \liminf_{n\to\infty}P(B|x^{(n)},a^{(n)})\ge P(B|x,a)
$ because for the model with backorders $P(x^*|x',a')=0$ for all
$x^*,x',a'\in \mathbb{R}$ in view of the continuity of the
distribution function $F_D$, and, for the model with {lost sales},
$P(x^*|x',a')=0$ for any $x',a'\in \mathbb{R}$ and $x^*\ne 0$, and
$P(0|x',a')=1-F_D(x'+a')$ is continuous in $(x',a')\in
\X\times\A$. Since $B$ is an arbitrary open set in $(\X,\rho_\X)$,
the stochastic kernel $P$ on $\X$ given $\X\times\A$ is weakly
continuous. Therefore,
$\limsup_{n\to\infty}P(B|x^{(n)},a^{(n)})\le P(B|x,a)$, for any
closed set $B$ in $(\X,\rho_\X)$.  Since any container $B_i$ is
simultaneously open and closed in $(\X,\rho_{\X})$, we have
$P(B_i|x^{(n)},a^{(n)}) \to P(B_i|x,a)$ as $n\to\infty$.

Set $\Psi(x)=x,$ if the inventory level $x$ belongs to a transparent container, and $\Psi(x)=b_i,$ if the inventory level belongs to a nontransparent
container $B_i$ with a label $i$. As follows from the definition of the metric $\rho_\X$, the function $\Psi:(\X,\rho_\X)\to (\Y,\rho)$ is continuous.
Therefore, the observation kernels $Q_0$  on $\Y$ given $\X$ and $Q$  on $\Y$ given $\A\times\X$, $Q_0(C|x):=Q(C|a,x):=\h\{\Psi(x)\in C\}$, $C\in \B(\Y)$,
$a\in\A$, $x\in\X$, are weakly continuous.

If all the containers are nontransparent, the observation set
$\Y=Y_L$ is countable,  and conditions of
Corollary~\ref{cor:inf-HL123} hold.  In particular, the function
$Q(b_i|a,x)=\h\{x\in B_i\}$ is continuous, if the metric $\rho_\X$
is considered on $\X.$ If some containers are transparent and some
are not, the conditions of Corollary~\ref{cor:main1} hold. To
verify this, we set $\Y_1:=\T$ and $\Y_2:=Y_L$ and note that
$\Y_2$ is countable  and the function $Q(b_i|x)={\bf I}\{x\in
B_i\}$ is continuous for each $b_i\in Y_L$ because $B_i$ is open
and closed in $(\X,\rho_\X).$ Note that $H(B|z,a,y)=P( B|y,a)$ for
any $B\in \B(\X)$, $C\in\B(\Y)$, $z\in\P(\X)$, $a\in \A$, and
$y\in\T$. The kernel $H$ is weakly continuous on
$\P(\X)\times\A\times \Y_1$. In addition, $\T=\cup_i B^t_i$, where
$B_i^t$ are transparent containers, is an open set in
$(\X,\rho_\X).$ Thus, if either Assumption ({\bf D}) or Assumption
({\bf P}) holds, then POMDP ($\X$, $\Y$, $\A$, $P$, $Q$, $c$)
satisfies the assumptions of Corollary~\ref{cor:main1}. Thus, for
the corresponding COMDP, there are stationary optimal policies,
optimal policies satisfy the optimality equations, and value
iterations converge to the optimal value.

The models studied in Bensoussan et al. \cite{Bens3, Bens4, Bens5} correspond to the partition $B_1=(-\infty,0]$ and $B_2=(0,+\infty)$   with the container
$B_2$ being nontransparent and with the container $B_1$ being either nontransparent (backordered amounts are not known \cite{Bens5}) or transparent (models
with lost sales \cite{Bens4},
 backorders are observable \cite{Bens3}). 
Note that, since $F_D$ is atomless, the probability that
$x_t+a_t-D_t=0$ is $0$, $t=1,2,\ldots\ .$

The model provided in this subsection is applicable to other inventory control problems, and the conclusions of Corollary~\ref{cor:main1} hold for them
too. For example, for problems with backorders, a nontransparent container $B_{0}=(-\infty,0)$ and a transparent container $B_1=[0, +\infty)$ model a
periodic review inventory control system for which nonnegative inventory levels are known, and, when the inventory level is negative,  it is known that
they are backorders, but their values are unknown.

\subsection{Markov Decision Model  with Incomplete
Information (MDMII) }\label{8.3} An MDMII is a particular version of a POMDP studied primarily before the POMDP model was introduced in its current formulation.  The reduction of MDMIIs with Borel state and action sets to MDPs was described  by  Rhenius~\cite{Rh} and
Yushkevich~\cite{Yu}; see also   Dynkin and  Yushkevich~\cite[Chapter 8]{DY}.  MDMIIs with  transition probabilities having densities  were studied by Rieder~\cite{Ri}; see also B\"auerle and Rieder~\cite[Part II]{BR}. An MDMII is
defined by an \textit{observed state space} $\Y$, an \textit{unobserved state space} $\W$, an \textit{action space} $\A$, nonempty \textit{sets of
available actions} $A(y),$ where $y\in\Y$,  a stochastic kernel $P$ on $\Y\times\W$ given $\Y\times\W\times\A$, and a one-step  cost function $c:\, G\to
\R,$ where $G=\{(y,w,a)\in \Y\times\W\times\A:\, a\in A(y)\}$ is the graph of the mapping $A(y,w)=A(y),$ $(y,w)\in \Y\times\W.$ Assume that:

(i) $\Y$, $\W$ and $\A$ are Borel subsets of Polish spaces. For
all $y\in \Y$ a nonempty Borel subset $A(y)$ of $\mathbb{A}$
represents the \textit{set of actions} available at $y;$

(ii) the graph  of the mapping $A:\Y\to 2^\A$, defined as $ {\rm Gr} ({A})=\{(y,a) \, : \,  y\in \Y, a\in A(y)\}$ is measurable, that is, ${\rm Gr}(A)\in
{\mathcal B}(\Y\times\A)$, and this graph allows a measurable selection, that is,  there exists a measurable mapping $\phi:\Y\to \mathbb{A}$ such that
$\phi(y)\in A(y)$ for all $y\in \Y$;

(iii) the transition kernel $P$ on $\X$ given $\Y\times\W\times\A$
is weakly continuous in $(y,w,a)\in \Y\times\W\times\A$;

(iv) the one-step cost $c$ is $\K$-inf-compact on $G$, that is,
for each compact set $K\subseteq\Y\times\W$ and for each
$\lambda\in \mathbb{R}$, the set ${\cal
D}_{K,c}(\lambda)=\{(y,w,a)\in G:\, c(y,w,a)\le\lambda\}$ is
compact.

Let us define $\X=\Y\times\W,$ and for $x=(y,w)\in \X$ let us
define $Q(C|x)= {\bf I}\{y\in C\}$ for all $C\in {\cal B}(\Y).$
Observe that this $Q$ corresponds to the  continuous function $y=
F(x),$ where $F(y,w)=y$ for all $x=(y,w)\in\X$ (here $F$ is a
projection of $\X=\Y\times\W$ on $\Y$).  Thus, as explained in
Example~\ref{exa1}, the stochastic kernel $Q(dy|x)$ is weakly
continuous in $x\in\X.$ Then by definition, an MDMII is a POMDP with
the state space $\X$, observation set $\Y$, action space $\A$,
available action sets $A(y)$, transition kernel $P$, observation
kernel $Q(dy|a,x):=Q(dy|x)$, and one-step cost function $c$.
However, this model differs from our basic definition of a POMDP
because action sets $A(y)$ depend on observations and one-step
costs $c(x,a)=c(y,w,a)$ are not defined when $a\notin A(y).$ To
avoid this difficulty, we set $c(y,w,a)=+\infty$ when $a\notin
A(y)$.  The extended function $c$ is $\K$-inf-compact on
$\X\times\A$ because the set ${\cal D}_{K,c}(\lambda)$ remains
unchanged for each $K\subseteq\Y\times\W$ and for each
$\lambda\in\mathbb{R}.$

Thus, an MDMII is a special case of a POMDP $(\X,\Y, \A,P,Q,c)$,
when $\X=\Y\times\W$ and observation kernels $Q$ and $Q_0$ are
defined by the projection of $\X$ on $\Y.$ The observation kernel
$Q(\,\cdot\,|x)$ is weakly continuous in $x\in \X$. As Example
\ref{exa1} demonstrates, in general this  is not sufficient for
the weak continuity of $q$ and therefore for the existence of
optimal policies.  The following example confirms this conclusion
for MDMIIs by demonstrating even the stronger assumption, that $P$
is setwise continuous, is not sufficient for the weak continuity
of the transition probability $q$.

\begin{example}\label{exa7.1} Setwise  continuity of a transition
probability $P$ on $\X$ given $\X\times\A$ for an MDMII is not sufficient for the weak continuity of the transition probability $q$ for the corresponding
COMDP. {\rm Set $\W=\{1,2\}$, $\Y=[0,1]$, $\X=\Y\times\W$, and $\A=\{0\}\cup\{\frac1n : {n=1,2,\dots}\}$. Let  $m$ be the Lebesgue measure on $\Y=[0,1]$
and $m^{(n)}$ be an absolutely continuous measure on $\Y=[0,1]$ with the density $f^{(n)}$ defined in (\ref{eqfnex1}).
As shown in Example~\ref{exa1}, the sequence of probability
measures $\{m^{(n)}\}_{n=1,2,\ldots}$ converges setwise to the
Lebesgue measure $m$ on $\Y=[0,1]$. Recall that $Q(C|a,y,w)={\bf
I}\{y\in C\}$ for $C\in\B(\Y)$. In this example, the setwise
continuous transition probability $P$ is chosen to satisfy the
following properties: $P(B|y,w,a)=P(B|w,a)$ for all $B\in\B(\X)$,
$y\in\Y,$ $w\in\W,$ $a\in\A$, that is, the transition
probabilities do not depend on observable states, and
$P(\Y\times\{w'\}|w,a)=0$, when $w'\ne w$ for all $w,w'\in\W$,
$a\in\A$, that is, the unobservable states do not change. For
$C\in \B(\Y)$, $w\in\W$, and $a\in \A$, we set
\[
P( C\times \{w\}|w,a)=\left\{
\begin{array}{ll}
m^{(n)}(C), &  \,w=2, \, a=\frac1n,\, n=1,2,\dots;\\
m(C),&\mbox{otherwise.}
\end{array}
\right.
\]
Fix $z\in\P(\X)$ defined by
\[
z(C\times \{w\})=0.5(\h\{w=1\}+\h\{w=2\})m(C),\quad w\in\W,\,
C\in\B(\Y).
\]
Direct calculations according to formulas (\ref{3.3})--(\ref{3.7})
imply that for $C,C'\in\B(\Y)$ and $w\in\W$
\[
R(C\times\{w\}\times C'|z,a)=
\left\{
\begin{array}{ll}
0.5 m^{(n)}(C\cap C'), &  \mbox{if }w=2\mbox{ and }a=\frac1n;\\
0.5m(C\cap C'), &  \mbox{otherwise:}
\end{array}
\right.
\]
which implies $R'(C'|z,\frac1n)=0.5 (m( C')+ m^{(n)}( C'))$,
$R'(C'|z,0)= m(C')$, and therefore we can choose 

\[H(C\times\{w\}|z,a,y)=\begin{cases} 0.5\,{\bf I}\{y\in  C\},
&{\rm if\ } a=0;\\{\bf I}\{y\in  C, f^{(n)}(y)=0\}+\frac{1}{3}{\bf
I}\{y\in C, f^{(n)}(y)=2\}, &{\rm if\ } w=1,\ a=\frac1n;\\
\frac{2}{3}{\bf I}\{y\in C, f^{(n)}(y)=2\}, &{\rm if\ } w=2,\ a=\frac1n;
\end{cases}
\]
%
%
%
where 
$ y\in \Y$ and $n=1,2,\ldots\ .$ The subset of atomic probability
measures on $\X$
\[
D:=\left\{z^{(y)}\in\P(\X)\,:\,  z^{(y)}( y,1)=\frac13, \ z^{(y)}(
y,2)=\frac23,\ y\in\Y\right\}
\]
is closed in $\P(\X)$. Indeed, an integral of  any bounded
continuous function $g$ on $\X$ with respect to a measure
$z^{(y)}\in D$ equals $\frac13g(y,1)+\frac23g(y,2)$, $y\in\Y.$
Therefore, a sequence $\{z^{(y^{(n)})}\}_{n=1,2,\ldots}$ of
measures from $D$ weakly converges to $z'\in\P(\X)$ if and only if
$y^{(n)}\to y\in\Y$ as $n\to\infty$ for some $y\in Y,$ and thus
$z'=z^{(y)}\in D.$ Since $D$ is a closed set in $\P(\X)$, if the
stochastic kernel $q$ on $\P(\X)$ given $\P(\X)\times A$ is weakly
continuous then $ \limsup_{n\to\infty} q(D|z,\frac1n)\le
q(D|z,0);$ Billingsley~\cite[Theorem 2.1(iii)]{Bil}. However,
$q(D|z,\frac1n)=z(f^{(n)}(y)=2)=0.5[m(f^{(n)}(y)=2)+m^{(n)}(f^{(n)}(y)=2)]=\frac34,$
$n=1,2,\ldots,$ and $q(D|z,0)=0.$ Thus, the stochastic kernel $q$
on $\P(\X)$ given $\P(\X)\times\A$ is not weakly
continuous.\EndPf}
\end{example}
%
%

Thus, the natural question is which conditions are needed for the
existence of optimal policies for the COMDP corresponding to an
MDMII? The first author of this paper learned about this question
from Alexander A. Yushkevich around the time when Yushkevich was
working on \cite{Yu}. The following theorem provides such a
condition.  For each open set $\oo$ in $\W$ and for any $C\in{\cal
B}(\Y)$, consider a family of functions
$\mathcal{P}^*_\oo=\{  (x,a)\to P(C\times\oo|x,a):\, C\in
\B(\Y)\}$ mapping $\X\times\A$ into $[0,1]$. Observe that
equicontinuity at all the points $(x,a)\in \X\times\A$ of the family
of functions $\mathcal{P}^*_\oo$ is a weaker assumption, than the
continuity of the stochastic kernel $P$ on $\X$ given $\X\times\A$
in the total variation.

\begin{theorem}\label{teor:Ren} Consider the  expected discounted cost criterion  with the discount factor $\alpha\in [0,1)$ and,
if the cost function $c$ is nonnegative, then $\alpha=1$ is also
allowed. If for each nonempty open set $\oo$ in $\W$ the family of
functions $\mathcal{P}^*_\oo$ is equicontinuous at all the points
$(x,a)\in \X\times\A$,
then the POMDP ($\X$,$\Y$,$\A$,$P$,$Q$,$c$) satisfies 
assumptions (a), (b), and (i) of Theorem~\ref{main1}, and
therefore the conclusions of that theorem hold.
\end{theorem}
\begin{proof}
Assumptions (a) and (b) of Theorem~\ref{main1} are obviously held,
and the rest of the proof verifies assumption (i).
From (\ref{3.3}) and (\ref{3.5}),
\[
R(C_1\times B\times C_2|z,a)=\int_{\X}P((C_1\cap C_2)\times
B|x,a)z(dx),\quad\ B\in \mathcal{B}(\W),\ C_1,C_2\in
\mathcal{B}(\Y),\ z\in\P(\X),\ a\in \A,
\]
\[
R'(C|z,a)=\int_{\X}P(C\times \W|x,a)z(dx),\qquad
 C\in \mathcal{B}(\Y),\ z\in\P(\X),\ a\in
\A.\] For any nonempty open sets $\oo_1$ in $\Y$ and $\oo_2$ in
$\W$ respectively, Theorem~\ref{kern}, with $\S_1 = \P(\X)$, $\S_2
= \X$, $\S_3 = \A$, $\oo = \X$, $\Psi(B | z) = z(B)$, and
$\mathcal{A}_0 = \{(x,a) \to P((\oo_1\cap C) \times \oo_2)|x,a): C
\in \B(\Y)\}$, implies the equicontinuity of the family of
functions
\[
\mathcal{R}_{\oo_1\times\oo_2}=\left\{(z,a)\to R(\oo_1\times
\oo_2\times C|z,a)\,:\, C\in\B(\Y)\right\},
\]
defined on $\P(\X) \times \A$, at all the points $(z,a) \in \P(\X)
\times \A$. Being applied to $\oo_1 = \Y$ and $\oo_2 = \W$, this
fact implies that the stochastic kernel $R'$ on $\Y$ given
$\P(\X)\times\A$ is continuous in the total variation. In
particular, the stochastic kernel $R'$ 
is setwise continuous.

Now, we show that Assumption {\bf {\bf(H)}} holds.  Since the metric spaces $\Y$ and $\W$ are separable, there exist countable bases $\tau_b^\Y$ and
$\tau_b^\W$ of the topologies for the separable metric spaces $\Y$ and $\W$, respectively. Then $\tau_b=\{\oo^\Y\times\oo^\W\,:\, \oo^\Y\in \tau_b^\Y,\,
\oo^\W\in \tau_b^\W\}$ is a countable base of the topology of  the separable metric space $\X=\Y\times\W$. Therefore, Assumption {\bf {\bf(H)}} follows
from  Lemma~\ref{R-H}, the equicontinuity of the family of functions $\mathcal{R}_{\oo_1\times\oo_2}$ for any open sets $\oo_1$ in $\Y$ and $\oo_2$ in
$\W$, and the property that, for any finite subset $N$ of $\{1,2,\ldots\}$,
\[\bigcap_{j \in N} (\oo^\Y_j \times \oo^\W_j) = ( \bigcap_{j \in N} \oo^\Y_j ) \times (\bigcap_{j \in N} \oo^\W_j)=\oo_1\times\oo_2, \qquad \oo^\Y_j \in \tau_b^\Y, \oo^\W_j \in \tau_b^\W \mbox{ for all } j \in N, \]
where $\oo_1=\cap_{j \in N} \oo^\Y_j$ and $\oo_2=\cap_{j \in N} \oo^\W_j$ are open subsets of $\Y$ and $\W,$ respectively.
\end{proof}

 {\bf Acknowledgements.} The first author thanks  Alexander A. Yushkevich 11/19/1930 -- 03/18/2012 who introduced
 him to the  theory of MDPs in general and in particular to the question on the existence of optimal policies for
 MDMIIs
  addressed in Theorem~\ref{teor:Ren}.  The authors thank  Huizhen (Janey)
  Yu, for providing an example of the weakly continuous kernels $P$ and
  $Q$ and discontinuous kernel $q$ mentioned before
  Example~\ref{exa1},  and M. Mandava and N.V. Zadoianchuk for their useful remarks. The research of the
first author was partially supported by NSF grants CMMI-0928490
and CMMI-1335296.

\end{document}